\documentclass[12pt,reqno]{amsart}
\usepackage{amssymb}
\usepackage{enumerate}

\usepackage{amsmath,amsthm}

\usepackage[margin=1.4in]{geometry}

\usepackage{todonotes}

\usepackage{hyperref,cleveref}

\title[Boundary reactions]{Minimizers for boundary reactions: renormalized energy, location of singularities, and applications
 }
\date{\today}

\author[X. Cabr\'e]{Xavier Cabr\'e}
\address{X. Cabr\'e \textsuperscript{1,2,3}
\newline
\textsuperscript{1} 
ICREA, Pg. Lluis Companys 23, 08010 Barcelona, Spain
\newline
\textsuperscript{2} 
Universitat Polit\`ecnica de Catalunya, Departament de Matem\`{a}tiques and IMTech, Av. Diagonal 647, 08028 Barcelona, Spain
\newline
\textsuperscript{3} 
Centre de Recerca Matem\`atica, Edifici C, Campus Bellaterra, 08193 Bellaterra, Spain.}
\email{xavier.cabre@upc.edu}

\author[N. Consul]{Neus C\'onsul}
\address{N. C\'onsul
\newline
Departament de Matem\`atica Aplicada I, Universitat Polit\`ecnica de Catalunya, 
Diagonal
647, 08028 Barcelona, Spain}
\email{neus.consul@upc.edu}

\author[M. Kurzke]{Matthias Kurzke}
\address{M. Kurzke
\newline
School of Mathematical Sciences, University of Nottingham, 
University Park, Nottingham NG7 2RD, United Kingdom}
\email{matthias.kurzke@nottingham.ac.uk}

\newtheorem{lemma}{Lemma}[section]
\newtheorem{theorem}[lemma]{Theorem}
\newtheorem{proposition}[lemma]{Proposition}
\newtheorem{corollary}[lemma]{Corollary}
\theoremstyle{definition}
\newtheorem{definition}[lemma]{Definition}
\newtheorem{remark}[lemma]{Remark}
\newtheorem{example}[lemma]{Example}

\def\Xint#1{\mathchoice
   {\XXint\displaystyle\textstyle{#1}}%
   {\XXint\textstyle\scriptstyle{#1}}%
   {\XXint\scriptstyle\scriptscriptstyle{#1}}%
   {\XXint\scriptscriptstyle\scriptscriptstyle{#1}}%
   \!\int}
\def\XXint#1#2#3{{\setbox0=\hbox{$#1{#2#3}{\int}$}
     \vcenter{\hbox{$#2#3$}}\kern-.5\wd0}}
\def\avint{\Xint-}

\newcommand{\NN}{{\mathbb N}}               
\newcommand{\RR}{{\mathbb R}}               
\newcommand{\R}{{\mathbb R}}               
\newcommand{\Ss}{{\mathbb S}}               

\newcommand{\ol}[1]{\overline{#1}}

\newcommand{\e}{{\varepsilon}}
\newcommand{\de}{{\partial}}
\newcommand{\dOm}{{\partial\Omega}}
\newcommand{\calA}{\mathcal{A}}
\newcommand{\calC}{\mathcal{C}}

\usepackage{xcolor}

\usepackage[notcite,notref]{}

\hypersetup{colorlinks=true, linkcolor=blue, citecolor=red}

\numberwithin{equation}{section}

\begin{document}

\thanks{Xavier Cabr\'e and Neus C\'onsul are supported by the Spanish grants PID2021-123903NB-I00 and RED2022-134784-T
funded by MCIN/AEI/10.13039/501100011033 and by ERDF “A way of making Europe”, and by the Catalan grant 2021-SGR-00087. This work is supported by the Spanish State Research Agency, through the Severo Ochoa and Mar\'{\i}a de Maeztu Program for Centers and Units of Excellence in R\&D (CEX2020-001084-M)}

\begin{abstract}
The Casten-Holland and Matano theorem for interior reactions states that no nonconstant stable solutions exist in convex domains  $\Omega$ of $\mathbb{R}^n$ under zero Neumann boundary conditions. In this paper we establish that the analogous statement fails for boundary reactions when $n=2$ (that is, for harmonic functions in~$\Omega$ with a Neumann reaction term on its boundary~$\partial\Omega$). For instance, nonconstant stable solutions exist when $\Omega$ is a square, or a smooth strictly convex approximation of it. In regular polygons of many sides, which approach the circle, we can prove the existence of as many nonconstant stable solutions as wished. Instead, in the circle such stable solutions do not exist.

More importantly, we can predict the existence or not of nonconstant stable solutions, as well as the location of its boundary ``vortices'' $(p,q)$, through the properties of a real function defined on $\partial\Omega\times\partial\Omega$ (the renormalized energy) which depends only on the conformal structure of the domain $\Omega$. This requires the development of a new Ginzburg-Landau theory for real-valued functions and the analysis of the half-Laplacian on the real line.
\end{abstract}

\maketitle

\section{Introduction}\label{sec:intro}

In the following subsection we describe a classical result on reaction-diffusion problems when the reaction takes place in the interior of a domain of $\R^n$. We turn later, as well as in the rest of the paper, to the analogue question when the reaction term takes place on the boundary of a planar domain.

\subsection{Interior reactions under zero Neumann boundary conditions}

The following is a  celebrated result from the late 1970s, proved by Casten and Holland, and independently by Matano. It establishes the nonexistence of nonconstant stable solutions
of the Neumann problem in bounded convex domains of $\R^n$.

\begin{theorem}[\cite{Casten:1978aa,Matano:1979aa}]\label{thm:Casten-Holland}
Let $\Omega\subset\R^n$ be a bounded, smooth, convex domain, and let $f\in C^1(\R)$. 
Assume $u\in C^2(\overline\Omega)$ is a solution of the Neumann problem
\begin{equation}\label{BVP-Neumann}
\left\{
\begin{array}{cl}
-\Delta u = f(u) & \text{in }\Omega \\[6pt]
\partial_{\nu} u = 0 & \text{on }\partial\Omega
\end{array}
\right.
\end{equation}
and that $u$ is stable, in the sense that
\[
\int_{\Omega} f'(u)\,\xi^2\,dx \le \int_{\Omega} |\nabla\xi|^2\,dx
\qquad\text{for all }\xi\in C^1(\overline\Omega).
\]

Then, $u$ is constant.
\end{theorem}

The proof of the result is simple.
To understand its statement, it is interesting to look at {\it the Allen-Cahn problem}
\begin{equation}\label{eq:pde_int}
\left\{
\begin{array}{cl}
- \Delta u = \dfrac{1}{\varepsilon^2} (u-u^3) & \text{in }\Omega \\[2mm]
\partial_{\nu} u = 0 & \text{on }\partial\Omega,
\end{array}
\right.
\end{equation}
which is the first variation of the functional 
\begin{equation}\label{int_ener}
\hat E_\varepsilon(u)=\int_\Omega\left( \frac12|\nabla u|^2 + \frac1{4\varepsilon^2} (1-u^2)^2\right)\, dx.
\end{equation}
Here the potential energy corresponds to a double-well potential. Critical points will satisfy $-1\leq u\leq 1$ by the maximum principle.

A well-known $\Gamma$-convergence result, \cite{Modica:1987,KohnSternberg:1989},  states that the rescaled energy $\varepsilon\hat E_\varepsilon$ converges, as $\varepsilon\to 0$,  to a multiple of the perimeter of the transition set where characteristic functions $u$ change from $-1$ to $1$. More precisely, for a universal constant $c>0$,
the limit energy is 
\begin{equation}\label{eq:energy_hat}
\hat{E}(u) =
\begin{cases}
c\,\mathrm{Per}_\Omega(\{ u = 1\}) & \text{if } u(x) = \pm 1 \text{ a.e. in }\Omega \\[1mm]
+\infty & \text{otherwise}.
\end{cases}
\end{equation}
To give some light into Theorem~\ref{thm:Casten-Holland}, note that given a strictly convex domain\footnote{\label{strict-convex}A domain of $\R^n$ is said to be strictly convex whenever two different points in the closure of the domain are given, the entire open segment between them lies in the interior of the domain.} in the plane $\R^2$ and any straight line segment (being the limiting transition layer in $\Omega$ where the solution goes from $-1$ to 1 ---we take it to be a straight line since it must minimize the perimeter or length functional) connecting two points on~$\partial\Omega$, it is always possible to find a shorter straight line connecting two other near-by points on the boundary, by the convexity of $\Omega$. This gives nonexistence of local minimizers for the limiting problem, suggesting that the same could be true for the Allen-Cahn problem for $\epsilon$ small. This fact is indeed true by Theorem~\ref{thm:Casten-Holland}.\footnote{The theorem admits some nice extensions to certain unbounded convex domains, as shown by Nordmann~\cite{Nordmann2021}.}

Instead in non-convex domains, for example in dumbbell-shaped ones, it is possible to construct nonconstant stable solutions of \eqref{eq:pde_int} near a local minimizer of 
the limiting functional, as done by Kohn and Sternberg~\cite{KohnSternberg:1989}. Note that in many dumbbell-shaped domains there will be a unique straight line segment connecting two opposite points in the thin neck with such straight segment being an isolated local minimizer of the length functional.\footnote{For a connection between unstable equilibria of the perimeter and 
solutions to \eqref{eq:pde_int} we refer to Jerrard--Sternberg \cite{Jerrard:2009ts}.}

\subsection{Boundary reaction problems}

Let us now turn into the case of {\it boundary reactions}. Here, the potential energy is computed on the boundary $\partial\Omega$ instead of~$\Omega$. Our functional is now
\begin{equation}\label{eq:ener}
E_{\varepsilon}(u) =  \int_{\Omega} \frac{1}{2}|\nabla u|^2 \, dx +
\int_{\partial \Omega} \dfrac{1}{\varepsilon} G(u) \, d\mathcal{H}^{n-1},
\end{equation} 
instead of \eqref{int_ener}, where $\varepsilon>0$ is a parameter. Denoting  $f=-G'$, the first variation of the functional is the boundary reaction problem
\begin{equation}\label{eq:pde}
\left\{
\begin{array}{cl}
\Delta u = 0 & \text{in }\Omega \\[6pt]
\partial_{\nu} u = \dfrac{1}{\varepsilon}f(u) & \text{on }\partial\Omega.
\end{array}
\right.
\end{equation}

A solution $u$ to \eqref{eq:pde} is said to be stable when
\[
\int_{\partial\Omega} \frac{1}{\e} f'(u)\,\xi^2\,d\mathcal{H}^{n-1}\le \int_{\Omega} |\nabla\xi|^2\,dx
\qquad\text{for all }\xi\in C^1(\overline\Omega).
\]
In particular, every local minimizer of the functional will be a stable solution.

Of particular interest, as in the interior reactions above, is the case of balanced bistable nonlinearities $f$, such as $f(u)=u-u^3$, corresponding to double-well potentials~$G$, such as $G(u)=(1-u^2)/4$. To describe some results for this problem, let us recall that in this case the functional is given by
\begin{equation}\label{eq:enerGL}
E_{\varepsilon}(u) =  \int_{\Omega} \frac{1}{2}|\nabla u|^2 \, dx +
\int_{\partial \Omega} \frac{1}{4\varepsilon}(1-u^2)^2 \, d\mathcal{H}^{n-1},
\end{equation} 
and the corresponding boundary reaction problem by
\begin{equation}\label{eq:pdeGL}
\left\{
\begin{array}{cl}
\Delta u = 0 & \text{in }\Omega \\[6pt]
\partial_{\nu} u = \dfrac{1}{\varepsilon}(u-u^3) & \text{on }\partial\Omega,
\end{array}
\right.
\end{equation}
for which solutions will satisfy $-1\leq u\leq 1$ by the maximum principle.

For this problem, the limit $\varepsilon\to 0$ is quite different than in the interior reaction case. Let us describe it in the case of interest in our paper: the planar case $n=2$.
By a 1994 result of Alberti, Bouchitt\'{e}, and Seppecher~\cite{ABS:1994,ABS:1998}, when $n=2$ the functionals $E_\varepsilon(u)/|\log\varepsilon|$ in \eqref{eq:enerGL} $\Gamma$-converge as $\varepsilon\to 0$ to the following limit functional:
\begin{equation}\label{eq:energy_limit}
\dfrac{E_\varepsilon(u)}{|\log\varepsilon|} \longrightarrow
\begin{cases}
\frac{2}{\pi}\,\mathcal{H}^0(\partial\{u=1\}) & \text{if } u(x) = \pm 1 \text{ a.e. on }\partial\Omega \\[1mm]
+\infty & \text{otherwise}.
\end{cases}
\end{equation}
Along the boundary, the transition sets (to go from $-1$ to 1) converge to \textit{points},
and the relevant $\Gamma$-limit only counts the number of points, which is discrete and hence does not have nontrivial local minima. Hence, the limiting problem gives no information about the existence, or not, of local minimizers. This is in contrast with the interior reaction case where existence (respectively, nonexistence), was obvious for the limiting problem in dumbbell-shaped domains (respectively, in strictly convex domains).

In 1996, the second author C\'onsul~\cite{Consul1996} proved the nonexistence of nonconstant stable solutions to \eqref{eq:pde} when the domain is a ball of $\R^n$, for any nonlinearity $f$ and parameter $\varepsilon>0$. On the other hand, the analogue of the interior dumbbell result was established in 1999 by C\'onsul and Sol\`a-Morales~\cite{ConsulSola-Mor:1999a}. They proved the existence of nonconstant stable solutions to \eqref{eq:pdeGL} when $\Omega$ are appropriate dumbbell domains in $\R^n$ and $\e >0$ is small enough. 

Still, the validity of the analogue result to the Casten-Holland and Matano result, Theorem~\ref{thm:Casten-Holland}, for boundary reactions in convex domains remained open. 
In those years there were at least two announcements that the analogue for the boundary reaction problem \eqref{eq:pde} in convex domains was also true. However, the proofs were incorrect. 

In fact, in 2005 C\'onsul and Jorba~\cite{Consul:2005aa} designed a finite element method together with a continuation-bifurcation procedure  to compute solutions to the boundary reaction problem \eqref{eq:pdeGL} with the bistable nonlinearity in the unit square $(0,1)^2$ of the plane. The method allowed to easily check, numerically, that the solutions making the transition near two mid points of opposite sides of the square are stable for $\varepsilon$ small enough, indeed for $\varepsilon < 0.352$. This was the origin of the current work.

The following is our first main result. We establish the existence of nonconstant stable solutions in a square for the bistable nonlinearity $(u-u^3)/\varepsilon$ when $ \varepsilon>0$ is small enough. These are the stable solutions found numerically in \cite{Consul:2005aa}. We can also find analogue stable solutions in smooth strictly convex domains approaching 
the square.\footnote{In all our results claiming the existence of a strictly convex domain (a notion defined in footnote~\ref{strict-convex}), we can further ensure the existence of a convex domain with positive curvature at all boundary points.}

\begin{theorem}\label{thm:squareREN}
Let  $\Omega$ be either the square $(0,1)\times (0, 1)$ in $\R^2$ or a smooth, bounded, and strictly convex domain sufficiently close to the square. 
Then, for $\varepsilon$ sufficiently small, there exists a nonconstant stable solution to problem \eqref{eq:pdeGL}. In addition, such solutions converge on $\partial\Omega$, as $\varepsilon\to 0$ and in the case of the square, to the function which is $-1$ on the left of the mid points $(1/2,0)$ and $(1/2,1)$ in opposite sides, and is $1$ on  the part of $\partial \Omega$ at their right.

The same result holds true in the rectangle $(0,L)\times(0,H)$ with $L\leq H$ and limiting jump points at $(L/2,0)$ and $(L/2,H)$.
\end{theorem}

The limiting points on the boundary where the jump is done will also be called {\em vortices} (as in complex Ginzburg-Landau theory) or singularities.
 
We announced this result in 2006 in several conferences. The paper in preparation was cited by Focardi and Garroni~\cite{FG:2007} back to 2007, and later by Gonz\'alez~\cite{Gonz:2009} and
Davila, del Pino, and Musso~\cite{DavilaPinoMusso:2011a}. In 2025, the paper~\cite{Cab:2025} by the first author Cabr\'e, to appear in the International Congress of Mathematicians (ICM) Proceedings 2026, stated Theorem~\ref{thm:squareREN} and described its proof.

Besides rectangles, we can also find planar domains in which there exist as many nonconstant stable solutions as wished. Perhaps surprisingly, these domains can be taken to approach the unit ball. Recall that in the unit ball there are no nonconstant stable solutions by the result of the second author, C\'onsul~\cite{Consul1996}.

\begin{theorem}\label{thm:T0}
For every $k\in \NN$, there exists a smooth, bounded, and strictly convex domain $\Omega$ of $\R^2$ in which, for $\varepsilon>0$ sufficiently small, there exists at least~$k$ different nonconstant stable solutions to problem \eqref{eq:pdeGL}. Such domains can be taken to have inradius and outradius (or circumradius) as close as wished to each other as $k\to\infty$. They do converge, in particular, to a ball.

The same result holds when $\Omega$ is a regular polygon with a sufficiently large number  of vertices, large enough depending on $k$.
\end{theorem}

Theorems \ref{thm:squareREN} and \ref{thm:T0} follow easily from a more important result, Theorem~\ref{thm:T1}. For every planar simply connected domain $\Omega\subset\R^2$, it establishes a criterion  which allows to ensure the existence of nonconstant stable solutions  for a sequence of parameters $\varepsilon$ tending to~0. In addition,
when such solutions exist, we can predict the location of the limiting boundary points $(p,q)$ of jump through the properties of a real valued function $W_\Omega=  W_\Omega (p,q)$ defined on $\partial\Omega\times\partial\Omega$. Such function depends only on the conformal structure of the domain $\Omega$ (or, equivalently, on its Green's function) and it is called {\it the renormalized energy} for problem \eqref{eq:pdeGL}. The proof of such result, stated next, requires the development of a new Ginzburg-Landau theory similar to the one of Bethuel-Brezis-Helein~\cite{BethuelBrezisHelein:1994a} for the complex valued equation. In our framework, the functions involved are now real valued and take $S^0=\{ -1,1\}$ values on the boundary~$\partial\Omega$. 

A related problem was studied by the third author Kurzke \cite{Kurzke:2006b}. He considered the energy
\begin{equation}\label{eq:introkurzke}
F_\e(u)=\frac12\int_\Omega |\nabla u|^2 dx dy + \frac1{2\e} \int_{\dOm} \sin^2(u-g)d\ell,
\end{equation}
where $\Omega\subset\R^2$ is bounded and simply connected and  $e^{ig}:\partial\Omega \to S^1$ is a map of degree $D\neq 0$. This is motivated
by applications to micromagnetics: if $e^{ig}$ is a unit tangent vectorfield and $m=e^{iu}$ is $S^1$ valued, then 
\[
F_\e(u) = \frac12\int_\Omega |\nabla m|^2 dx dy+ \frac1{2\e} \int_{\dOm} (m\cdot \nu)^2 d\ell,
\]
where the boundary integral comes from an expansion of the nonlocal magnetostatic energy. 

The introduction of $g$ acts as a forcing term, and since $D\neq 0$ the functional $F_\e$ has nontrivial minimizers. These converge
as $\e\to 0$ to a harmonic function $u_*$ with $\sin^2(u_*-g)=0$ on the boundary that jump by $\pi$ at exactly $2D$ points on the boundary. 
These points can be found as the global minimizers of a renormalized energy that depends on $g$ as well as on $\Omega$.

Going back to our problem, to state Theorem~\ref{thm:T1} we must introduce some objects. Given two points $p$ and $q$ on $\partial\Omega$, $p\neq q$, let $I_{p,q}$ be the connected component of $\partial\Omega\setminus\{p,q\}$ corresponding to going from $p$ to $q$ counterclockwise. We define $\chi^{p,q}$ as the characteristic  function on the boundary given by
\begin{equation}\label{eq:chif}
\chi^{p,q} =
\begin{cases}
-1 & \text{ on } I_{p,q}  \\
1 & \text{ on } \partial\Omega\setminus  I_{p,q}.
\end{cases}
\end{equation}
Note that this function
$\chi^{p,q}$ minimizes the boundary term in the energy \eqref{eq:enerGL} since $(1-(\chi^{p,q})^2)^2=0$ on all $\partial\Omega$. Hence, it could be natural
to guess that a minimizer of the total energy would be obtained by the harmonic extension to
$\Omega$ of $\chi^{p,q}$, that is, $u_0^{p,q}$ satisfying
\begin{equation}\label{hext}
\left\{ \begin{array}{rcll} \Delta u_0^{p,q} & = & 0 &\mbox{ in } \Omega \, , \\
u_0^{p,q}& = & \chi ^{p,q}  & \mbox{ on } \partial\Omega\, .
\end{array} \right.
\end{equation}
Notice that since $\chi^{p,q}\not\in H^{1/2}(\partial\Omega)$, we have $u_0^{p,q}\not\in H^1(\Omega)$ and thus it
can not be a minimizer of the energy.\footnote{It can be shown that $\chi^{p,q}\in
H^r(\dOm)$, $r<1/2$ and, therefore $u_0^{p,q} \in H^s(\Omega)$, $s<1$. It is also easy to see $u_0^{p,q}$ is unique, for example by using a removal of singularities argument. We have that $-1\leq u_0^{p,q}\leq 1$ and  $u_0^{p,q}\in C^{\infty} (\Omega)\cap C^0(\overline{\Omega}\setminus \{p,q\})$. 
}

Even though the Dirichlet energy of $u_0^{p,q}$ is infinite, we can still study its dependence on the choice of points $p,q$ by renormalizing appropriately.
More precisely, as in the complex valued Ginzburg-Landau theory, we can remove small balls around $p$ and $q$ and study the behavior of the Dirichlet energy on $\Omega_\rho=\Omega\setminus (B_\rho(p) \cup
B_\rho(q))$ as $\rho\to 0$. In Theorem~\ref{thm:calcW1} we will prove that
\begin{equation}\label{eq_ener2}
\frac12\int_{\Omega_{\rho}} |\nabla u_0^{p,q}|^2 \, dxdy = \frac 4\pi  \log \frac{1}{\rho} +
W_{\Omega}(p,q) + O(\rho), 
\end{equation}
where $W_\Omega$ is called the {\it renormalized energy} of the problem and can be computed by means of the conformal map from $\Omega$ to the half-plane  (see Definition~\ref{defi:renen}) or from Green's functions of~$\Omega$ (see Proposition~\ref{propo:altren}). Indeed, in Section~\ref{sec:renen} we will see that
\begin{align}\label{renorm}
W_{\Omega}(p,q)&=-\frac{2}{\pi} \log \frac{\partial ^2 G^D}{\partial \nu_p
\partial \nu_q}(p,q)\notag \\
&=2\left( R^N(p,q)+R^N(q,q)-2 G^N(p,q)\right)\\
&= {\frac{4}{\pi} \log |p-q| + 2 \left( R^N(p,p)+R^N(q,q)-2R^N(p,q)\right) }, \notag
\end{align}
where $G^D(p,q)$ is the Dirichlet Green's function of
$\Omega$, $G^N(p,q)$ is the Neumann Green's function of $\Omega$, and $R^N(p,q)$ is the regular
part of the last one.

We can now state our main result. It holds for all bistable balanced nonlinearities~$f$. More precisely, throughout the paper, for the potential $G$ (recall that $f=-G'$) we will assume that, for some $\alpha\in (0,1)$,
\begin{equation}\label{potential}
G\in C^{2,\alpha}([-1,1]), \,\, G\ge 0, \,\, G(t)=0 \text{ only for } t=\pm 1,\text{ and }G''(\pm 1)>0.
\end{equation} 
The classical potential with these
properties is $G(t)=\frac14(1-t^2)^2$; here $f(t)=t-t^3$. A different example is 
$G(t)=\frac{2}{\pi^2} \cos^2(\frac{\pi t}{2})$; here $f(t)=\frac{1}{\pi}\sin{\pi t}$. In particular,  there exists a number $t_*\in (0,1)$ such that 
\begin{equation}\label{convexG}
G \text{ is convex in the two connected components of } \{ t_*\le |t|\le 1\}.
\end{equation}
Notice that we allow for non-even potentials $G$.

\begin{theorem}\label{thm:T1}
Let $f=-G'$ satisfy \eqref{potential}. Let $\Omega$ be a smooth, bounded, and simply connected domain of $\R^2$ such that the renormalized energy $W_\Omega$ admits an isolated local minimizer $(p,q)\in \partial\Omega\times\partial\Omega$.

Then, there is an $\e_0>0$  for which there exists a nonconstant stable solution 
$u_\e$ to \eqref{eq:pde} for every $\varepsilon\in (0,\varepsilon_0)$. 
Moreover, these solutions satisfy $u_\e\to \chi^{p,q}$ in $L^2(\dOm)$, 
where $(p,q)$ is the given isolated local minimizer of $W_\Omega$.
\end{theorem}

Notice that $W_\Omega(p,q)\to -\infty$ if $q\to p$. However, for certain domains $W_\Omega$ may also have isolated local minimizers. The simplest example in which this occurs is a square and $(p,q)$ are centered points in opposite sides. In polygons with more sides, a larger number of isolated local minimizers exist. Thus, we see that Theorems~\ref{thm:squareREN} and \ref{thm:T0} follow easily from Theorem~\ref{thm:T1}.

The proof of Theorem \ref{thm:T1} requires to prove sharp upper and lower bounds for the energy of the following constrained minimizers. They will allow to construct stable solutions from the existence of an isolated minimizer $(p, q)$ of the renormalized energy. Given such minimizing points $p, q$ on $\partial \Omega$, let $u_\e$ be a minimizer of the functional $E_\e$ among functions having boundary values in a small {\it closed} $L^2(\partial\Omega)$ neighborhood of~$\chi^{p,q}$. Then, the lower bound reads as
\begin{equation}\label{lowerb}
E_\e(u_\e) \ge \frac4\pi \log\frac1\e + W_\Omega(p,q) +2C_f -o_\e(1),
\end{equation}
while the upper bound (which is much simpler to prove) is the same with $-o_\e(1)$ replaced by $+o_\e(1)$. Here, $C_f$ is a constant depending only on the nonlinearity $f$.

As a byproduct of this analysis, we also have the following result in the spirit of $\Gamma$-convergence. It connects the energy $E_\e(u)$ with the renormalized energy $W_\Omega(p,q)$. 

\begin{theorem}\label{thm:gammac}
Let $f=-G'$ satisfy \eqref{potential}. Let $\Omega\subset\R^2$ be any smooth, bounded, and simply connected domain, and $p\neq q$ two arbitrary points on $\partial\Omega$. Then,
\begin{enumerate}
\item[(a)] Let $(w_\e)$ be functions in $H^1(\Omega)$ such that the traces satisfy $w_\e\to \chi^{p,q}$ in $L^2(\dOm)$. 
Then as $\e\to 0$,
\[
E_\e(w_\e) \ge \frac4\pi \log\frac1\e + W_\Omega(p,q) +2C_f -o_\e(1).
\]
\item[(b)] There exist functions $v_\e\in H^1(\Omega)$ such that $v_\e \to \chi^{p,q}$ in $L^2(\dOm)$ and
\[
E_\e(v_\e) \le \frac4\pi \log\frac1\e + W_\Omega(p,q) +2C_f +o_\e(1)
\]
as $\e\to 0$.
\end{enumerate}
\end{theorem}

To establish \eqref{lowerb}, and also the upper bound, one blows-up problem \eqref{eq:pde} near the point $p$, obtaining the boundary reaction problem 
\begin{equation}\label{eq:halfhalf}
\left\{
\begin{array}{cl}
\Delta u = 0 & \text{in }\R^2_+:=\{x\in\R, y>0\} \\[6pt]
\partial_{\nu} u = f(u) & \text{on }\partial\R^2_+
\end{array}
\right.
\end{equation}
in the half-plane, where $f=-G'$ satisfies \eqref{potential}. Note that the parameter $\varepsilon$ does not appear anymore. Notice that, with $v=u(\cdot,0)$, problem \eqref{eq:halfhalf} is equivalent to the equation
\begin{equation*}\label{half-lapl}
(-\Delta)^{1/2} v = f(v) \quad\text{in } \R
\end{equation*}
for the half-Laplacian.

For our pourposes, we will need a classification result for solutions to problem  \eqref{eq:halfhalf} that we could not find in the literature, except for $f(u)=\frac{1}{\pi}\sin(\pi u)$ by a classical result of Toland~\cite{Toland:1997}.\footnote{\label{fn:Toland}For the half-plane problem \eqref{eq:halfhalf}, using complex variables Amick and Toland~\cite{AmickTol:1991} (for $f(u)=-u+u^2$, $u>0$,  i.e., the Benjamin-Ono equation in hydrodynamics; notice that this nonlinearity is not a bistable one) and Toland~\cite{Toland:1997} (for $f(u)=\frac{1}{\pi}\sin(\pi u)$, i.e., the Peierls-Nabarro equation in crystal dislocations) were able to find explicit expressions for all their bounded solutions. It turns out that these two equations (and for these particular choices of nonlinearities) are ``completely integrable systems'' and together they form a Lax pair.} In particular, our result is new even for $f(u)=u-u^3$.

\begin{theorem}[]\label{thm:no_homoclinic}
Let $f=-G'$ satisfy \eqref{potential}.
Let $u$ be a solution to \eqref{eq:halfhalf} such that $-1\leq u\leq 1$ and $u(x,0)\to -1$ as $x\to \pm \infty$ . Then, $u\equiv -1$.
\end{theorem}

In particular, this result prevents minimizers of $E_\e$ to make an homoclinic oscillation from $-1$ to a certain value in $(-t_*, 1)$ and back to the value $-1$ near in an open set with length of order $\e$.

We briefly comment on related works. For the energy with forcing \eqref{eq:introkurzke}, a $\Gamma$-convergence result 
in the spirit of Alberti-Bochitt\'e-Seppecher \cite{ABS:1994} was shown by Kurzke~\cite{Kurzke:2006a}, followed by
an analysis of the energy of minimizers and some critical points~\cite{Kurzke:2006b} and of the relevant gradient flow~\cite{Kurzke:2007a}. In a forthcoming paper~\cite{BEK19}, convergence and energy expansion results for critical points are shown assuming only a natural energy bound.
The PDE-based analysis of minimizers was generalized to a full second order $\Gamma$-expansion of the energy by Ignat and Kurzke~\cite{IK_jac},
who also rigorously connected the problem to a thin film limit of micromagnetics~\cite{IK_bdv}.  

There are also works generalising the work of Bethuel-Brezis-H\'elein to situations allowing for interior and boundary vortices; we mention 
as representative examples
Moser~\cite{Moser2003}, who studied a problem motivated by micromagnetics and Alama-Bronsard-Golovaty~\cite{ABG2020}, who 
studied nematic liquid crystals.

It is also of interest the paper by Davila, del Pino, and Musso~\cite{DavilaPinoMusso:2011a}, where they construct
solutions that develop multiple transitions from $-1$ to 1 and vice-versa along a
connected component of the boundary of an arbitrary bounded smooth domain of~$\R^2$. Such solutions will often be unstable. 

\subsection{Proof of the main theorems}\label{sec:proofmainth}

In this subsection we collect our central intermediate results and show how to combine them to
prove two of the main theorems stated above. Indeed, at the end of this subsection we include the proofs of Theorems~\ref{thm:T1} and \ref{thm:gammac}. Detailed proofs of the intermediate results are given in later sections. 

Let us first discuss the existence of local minimizers to the renormalized energy. In a square,
this result (Theorem~\ref{thm:squareREN}) follows  by a direct computation using the series
representation of the Green's function; see Theorem~\ref{thm:rectangle}.
The proof of Theorem~\ref{thm:T0}, i.e., the existence of domains where $W_\Omega$ has large numbers of 
nontrivial local minimizers, is a computation implementing the idea that this holds true in polygons and stays correct for small perturbations; see Theorem~\ref{thm:lomi}.

For the proof of Theorem~\ref{thm:T1} (which uses all ingredients of the paper), we take a simply connected bounded domain which admits an isolated local minimizer $(p,q)\in \partial\Omega\times\partial\Omega$ of $W_\Omega$. 

We  first construct functions that are near-optimal in energy without necessarily being solutions of the PDE. This is the so-called upper bound.

\begin{proposition}[proved later as Proposition ~\ref{prop:ub}]\label{prop:ub_intro}
For any $p\neq q\in\dOm$, there exist functions $v_\e$ such that $-1\le v_\e\le 1$, $v_\e\to \chi^{p,q}$ in $L^2(\dOm)$ and 
\begin{equation}\label{eq:ub_intro}
E_\e(v_\e) \le \frac4\pi \log \frac1\e + W_\Omega(p,q) + 2 C_f + o_\e(1).
\end{equation}
\end{proposition}

Next, we try to construct functions $u_\e$ that solve the PDE \eqref{eq:pde} and converge to the characteristic function $\chi^{p,q}$ on the boundary. 
For technical reasons, we replace $\chi^{p,q}$ by a smoothed version $\chi^{p,q}_\e$ in Lemma~\ref{lem42}. 
We now consider for a small $a>0$ the sets
\begin{equation}\label{neigh}
\calC_a^\e = \{ u\in H^1(\Omega) : \|u-\chi^{p,q}_\e \|_{L^2(\dOm)}^2 \le a^2 \}.
\end{equation}
Note that the set $\calC_a^\e$ is closed under weak $H^1$ convergence. We can thus minimize
$E_\e$ among $u\in \calC_a^\e$ and find a minimizer $u_\e$. It solves the following problem.

\begin{proposition}[proved later as Proposition~\ref{pro:conmin}]\label{pro:conmin_intro}
Any minimizer $u_\e$ of $E_\e$ over the closed neighborhood $\calC_a^\e$ in \eqref{neigh} satisfies \begin{equation}\label{eq:elelam_intro}
\begin{cases}
\Delta u_\e &=0\qquad\qquad\qquad\qquad\qquad\text{in $\Omega$,}\\
\frac{\partial u_\e}{\partial \nu} &=\frac1\e f(u_\e)+ \lambda_\e (u_\e-\chi^{p,q}_\e) \quad \text{on $\dOm$},
\end{cases}\end{equation}
for some $\lambda_\e\in\R$ with $|\lambda_\e|\le C\sqrt\frac{|\log\e|}{\e}$.
Solutions to \eqref{eq:elelam_intro} are smooth up to the boundary and satisfy 
$|\nabla u_\e|\le \frac C\e.$
\end{proposition}
The Lagrange multiplier $\lambda_\e$ will only be nonzero (i.e., active) when the minimizer $u_\e$ lies on the 
boundary of the set $\calC_a^\e$. A central point in our proof is to show that, if $(p,q)$ is an isolated local
 minimizer of the renormalized energy, then it is possible to make a suitable choice of $a$ which forces the minimizer
 $u_\e$ to lie in the interior of $\calC_a^\e$. We will then have $\lambda_\e=0$,
which makes \eqref{eq:elelam_intro} become our true problem \eqref{eq:pde} and we will have found a stable solution to it. To accomplish this we need numerous intermediate results. We describe the main ones next.

We mainly aim to prove a lower bound for the  minimizers $u_\e$ in the closed neighborhood~$\calC_a^\e$
that complements the upper bound \eqref{eq:ub_intro}. This proceeds in 
several steps, first by considering the set $S_\e=\{ x\in \dOm: |u_\e|\le t_*\}$. In relation to $t_*$, recall \eqref{convexG}. 

\begin{proposition}[see Section~\ref{sec:poho}, especially Proposition~\ref{pro:57}]
Let $(u_\e)$ be a sequence of solutions to \eqref{eq:elelam_intro} with $E_\e(u_\e)\le K|\log\e|$ for some constant $K$ independent of $\e$.

Then, for a subsequence of parameters $\e$, there exists $M>0$ (independent of~$\e$) such that $S_\e$ can be covered by a uniformly bounded number of disjoint balls of radius $M\e$. In addition, the mutual distance between any two of such balls, divided by $\e$, tend to $\infty$.
\end{proposition}

In fact, we later show that only two of such balls are sufficient for minimizers in $\calC_a^\e$.

\begin{proposition}[proved as Proposition~\ref{pro:just2}]
If $u_\e$ is a sequence of  minimizers in the closed neighborhood $\calC_a^\e$, then for a subsequence, the 
 set $S_\e$ can be covered with exactly two balls of radius $M\e$, with $M$ independent of $\e$.
\end{proposition}

Taking another subsequence, we obtain the following convergence result.

\begin{proposition}[proved as Corollary~\ref{cor:pq}]\label{corpqintro}
Let $u_\e$ be  minimizers in the closed neighborhood $\calC_a^\e$. Then, for a subsequence, we have $u_\e \to \chi^{p_0,q_0}$ in $L^2(\dOm)$, where $p_0,q_0\in \dOm$ are points with
$|p_0-p|+|q_0-q|\le Ca$ for some constant $C$ independent of $\e$.
\end{proposition}

Using that $|\lambda_\e|\le C\sqrt{|\log\e|/\e}$, we have that the $\e$-scale blow-up of \eqref{eq:elelam_intro} is a solution~to  
\begin{equation}\label{eq:layerpde_intro}
\begin{cases}
\Delta V &=0  \quad\qquad\text{in $\RR^2_+=\{(x,y) : y>0\}$} \\
\frac{\partial V}{\partial \nu} &= f(V) \quad\text{ on $\partial\RR^2_+=\RR$}.
\end{cases}
\end{equation}

We perform this blow-up in Proposition~\ref{pro:61} after a conformal transformation sending~$\Omega$ to a half-plane. In this subsection, for simplicity of exposition, we will ignore this transformation.

We then need the following new classification result for solutions of \eqref{eq:layerpde_intro} 
with limits at $\pm\infty$. In particular, we show that every homoclinic solution $V$ (i.e., 
a solution with $V(x,0)$ tending to the same limit in $\{-1,1\}$ at $-\infty$ and at $+\infty$) must be a constant. 
In the following, $U$ is the unique (up to translation) ``layer" solution to \eqref{eq:layerpde_intro}, i.e., a solution with $U_x>0$ and $U(x,0)\to \pm 1$ as $x\to \pm \infty$;
see Theorem~\ref{thm:layerprops}. 

\begin{theorem}[proved as Theorem~\ref{thm:uni} and Proposition~\ref{prop:i_er}]\label{thm:uni_intro}
Let $V$ be a solution to \eqref{eq:layerpde_intro} with values in $[-1,1]$ and such that $|V(x,0)|\to 1$ as $x\to \pm \infty$ . Then,
either $V$ is constant $\pm 1$ or a layer solution, i.e., there exists $b\in \R$ with $V(x,y)=U(\pm (x-b),y)$.

In addition, writing $B_R^+=\{|(x,y)|<R, y>0\}$ and  $I_R=[-R,R]$, the 
energy of a layer solution $U$ satisfies
\[
\frac12\int_{B_R^+ }  |\nabla U|^2 \, dxdy + \int_{I_R} G(U)\, d\ell =\frac2\pi \log R+C_f +o(1)
\] 
as $R\to\infty$,
where $C_f$ is a constant depending only on the nonlinearity $f$.
\end{theorem}

This result is central for obtaining lower bounds for the energy of the minimizers~$u_\e$ in the closed neighborhood $\calC_a^\e$.
In balls of radius $\rho$ around  the points $p_0$ and $q_0$ found in Proposition~\ref{corpqintro}, the function $u_\e$ is close to 
a rescaling of a layer solution. With some additional estimates (see Propositions~\ref{prop:inner} and~\ref{prop:middle}), this 
yields a sharp lower bound:
the energy of 
 $u_\e$ in $B_\rho(p_0)\cap \Omega$ (or in $B_\rho(q_0)\cap \Omega$) is 
  bounded below as follows
  \begin{equation}\label{new-lower}
    \frac12 \int_{B_\rho(p_0)\cap \Omega} |\nabla u_\e|^2 dx dy + \int_{B_\rho(p_0)\cap \partial\Omega} \frac{1}{\e} G(u_\e)\, d\ell
    \geq  \frac{2}{\pi} \log \frac\rho\e +C_f -o_\rho(1), 
  \end{equation}
which is unbounded as $\e\to 0$.
  As $u_\e$ is a minimizer, it satisfies the energy upper bound \eqref{eq:ub_intro}, that is,
\[
E_\e(u_\e) \le \frac4\pi \log \frac1\e + W_\Omega(p,q) + 2 C_f + o_\e(1).
\]
As the divergent part in this upper bound matches exactly with the one in the local lower bound \eqref{new-lower},  writing $\Omega_\rho=\Omega\setminus (B_\rho(p_0)\cup B_\rho(q_0))$, we find that
\[
\frac12 \int_{\Omega_\rho} |\nabla u_\e|^2 dx dy\le \frac4\pi \log \frac1\rho+ C.
\] 
This estimate gives $H^1$ convergence away from the points $p_0$ and $q_0$, i.e., in $\Omega_\rho$ as above for every $\rho>0$.

The weak lower semicontinuity of the Dirichlet integral together with the convergence $u_\e\to \chi^{p_0,q_0}$ 
  on the boundary (Proposition~\ref{corpqintro}) lead to 
  \[
\liminf_{\e\to 0}  \frac12\int_{\Omega_\rho} |\nabla u_\e|^2 \,dx dy \ge 
\frac12\int_{\Omega_\rho} |\nabla u_0^{p_0,q_0}|^2 \,dxdy,
\]
and using \eqref{eq_ener2} or Theorem~\ref{thm:calcW1} we find 
 \[
\liminf_{\e\to 0}  \frac12\int_{\Omega_\rho} |\nabla u_\e|^2 \,dx dy \ge 
\frac4\pi \log\frac1\rho + W_\Omega(p_0,q_0)+O(\rho).
  \]
Together with the sharp local lower bound on the energy near $p_0$ and $q_0$, stated in \eqref{new-lower}, we deduce the following sharp lower bound
for the energy in all $\Omega$.

\begin{proposition}[proved as Proposition~\ref{pro:finallb}]\label{pro:finallb_intro}
Let $u_\e$ be  minimizers in the closed neighborhood $\calC_a^\e$. We have
\begin{equation}\label{eq:lowerb_intro}
E_\e(u_\e) \ge \frac4\pi \log\frac1\e + W_\Omega(p_0,q_0)+ 2C_f -o_\e(1)
\end{equation}
as $\e\to 0$.
\end{proposition}

With these preparations, the proofs of the main results are now very short.

\begin{proof}[Proof of Theorem~\ref{thm:T1}]
Let $(p,q)$ be an isolated local mimimizer of $W_\Omega$. 
Let $u_\e$ be as in Proposition~\ref{pro:conmin_intro}, i.e., minimizing $E_\e$ as defined in \eqref{eq:ener} in the closed neighborhood $\calC_a^\e=\{ u\in H^1(\Omega): \|u-\chi_\e^{p,q}\|^2_{L^2(\dOm)}\le a^2\}$.

We take subsequences such the results above
can be applied to $u_\e$. 
In particular, we have that $u_\e\to \chi^{p_0,q_0}$ in $L^2(\dOm)$ for some $p_0,q_0\in \dOm$ by Proposition~\ref{corpqintro}.
Comparing the lower bound \eqref{eq:lowerb_intro} and the upper bound \eqref{eq:ub_intro} (note that $E_\e(u_\e)\leq E_\e(v_\e)$ by minimality), we find
\[
W_\Omega(p_0,q_0)\le W_\Omega(p,q),
\]
where $(p,q)$ is the isolated minimizer of $W_\Omega$.
From Proposition~\ref{corpqintro}, $|p-p_0|+|q-q_0|\le Ca$.
Choosing $a$ small enough, the fact that $(p,q)$ are is an isolated  local minimizer leads to $p=p_0$ and $q=q_0$.
However, this means that $\|u_\e-\chi^{p,q}_\e\|_{L^2(\dOm)}\to 0$. Thus, for $\e$ small enough, we have $\|u_\e-\chi^{p,q}_\e\|_{L^2(\dOm)}<a^2$.
This gives that the Lagrange multiplier in 
\eqref{eq:elelam_intro} is not active and $u_\e$ is a nonconstant solution of \eqref{eq:pde} as claimed. In addition, being a local minimizer, $u_\e$ is stable.
\end{proof}

\begin{proof}[Proof of Theorem~\ref{thm:gammac}]
The upper bound is just Proposition~\ref{prop:ub_intro}.

We now consider the lower bound. If $p\neq q$ are arbitrary points on $\dOm$ and
$(w_\e)$ are functions in $H^1(\Omega)$ such that the traces satisfy $w_\e\to \chi^{p,q}$ in $L^2(\partial \Omega)$, then
 for every $a>0$ there is $\e_0>0$ such that for $0<\e<\e_0$,
   $w_\e\in \calC_a^\e$. This means that we can compare the energy of $w_\e$ to that of a minimizer $u_\e$ of 
   $E_\e$ over the closed neighborhood $\calC_a^\e$ (as in Proposition~\ref{pro:conmin_intro})
   and obtain 
   \[
   E_\e(w_\e)\ge E_\e(u_\e). 
   \] 
By Propositions~\ref{corpqintro} and~\ref{pro:finallb_intro}, we find $p_0,q_0\in \dOm$ 
with $|p_0-p|+|q_0-q|\leq Ca$ and 
\[
E_\e(u_\e)\ge \frac4\pi \log\frac1\e + 
W_\Omega(p_0,q_0)
+2C_f -o_\e(1).
\]
Letting $a\to 0$, we have $p_0\to p$ and $q_0\to q$. Hence, $W_\Omega(p_0,q_0)\to W_\Omega(p,q)$
and we conclude the result. 
\end{proof}

\subsection{Plan of the paper}
To make the reading easier and since the section titles are quite explanatory, we simply list them here.
\begin{itemize}
\item[2.] The layer solution in the half-plane. Nonexistence of homoclinic solutions.
\item[3.] Renormalized energy for boundary values in $\Ss^0$.
\item[4.] Minimization in a closed neighborhood and energy upper bounds.
\item[5.] Poho\u{z}aev balls and covering arguments.
\item[6.] Blow-up and the sharp lower bound.
\item[7.]  Domains where the  renormalized energy has nontrivial local minimizers. 
\item[A.] Appendix. Computing the value of $C_f$ in a special case.
\end{itemize}

\section{The layer solution in the half-plane. Nonexistence of homoclinic solutions}\label{sec:layer}

In this section we study a natural blow-up problem related to \eqref{eq:pde}. In the half-plane $\R^2_+=\{ (x,y)\in \R^2: y>0\}$, we consider 
\begin{equation*}
\begin{cases}
\Delta U &=0 \quad \text{ in } \R^2_+\\
\frac{\de U }{\de \nu} & = f(U) \quad\text{on }\partial\R^2_+=\R\times\{0\},
\end{cases}
\end{equation*}
where we recall that $f=-G'$ satisfies \eqref{potential}.

We will classify the solutions with $|U(x,0)|\to 1$ as $x\to \pm \infty$, and will show that they are either constant or \emph{layer solutions}.
A complete classification of solutions (including all periodic solutions) is only known for the Peierls-Nabarro nonlinearity $f(t)=\frac{1}{\pi} \sin(\pi t)$ by results of~\cite{ Toland:1997}; see footnote~\ref{fn:Toland}.

We start with two simple results.
\begin{lemma}\label{lem5p5}
Let $U$ and $V$ be harmonic in $\R^2_+$. 
Then for every  $\xi \in C^\infty(\ol{\R^2_+})$, the quantity
\[
D=\int_{\R^2_+} \xi^2 |\nabla (U-V)|^2 \, dx dy
\]
satisfies
\[
D\le 2D^{1/2} \left( \int_{\R^2_+} |\nabla \xi|^2 |U-V|^2 \,dxdy\right)^{1/2} + \int_\R \xi^2 \left( \frac{\de U}{\de\nu}-\frac{\de V}{\de \nu}\right)(U-V)\, dx
\]
\end{lemma}
\begin{proof}
Integrating by parts,
\[
D = \int_{\R^2_+} 2 \xi (U-V) \nabla (U-V)\cdot \nabla \xi \, dxdy + \int_\R \xi^2 \left( \frac{\de U}{\de\nu}-\frac{\de V}{\de \nu}\right)(U-V)\, dx.
\]
We immediately obtain the claim by using the Cauchy-Schwarz inequality.
\end{proof}

We introduce some notation.
We let $\theta^{x_0}:\R^2_+\to (0,\pi)$ denote the argument function with respect to $x_0\in\R$, more precisely we set 
\[
\theta^{x_0}(x,y)=\frac\pi 2-\arctan\frac{x-x_0}y.
\]
Then $\theta^{x_0}$ is an argument function in the sense that
\[e^{i\theta^{x_0}(x,y)}=x+iy\]
for $(x,y)\in \R^2_+$. Moreover, we have
\[\nabla \theta^{x_0}(x,y)=\frac{1}{(x-x_0)^2+y^2}\begin{pmatrix} y \\ -(x-x_0)
\end{pmatrix}, \qquad |\nabla \theta^{x_0}(x,y)|=\frac1{\sqrt{(x-x_0)^2+y^2}}.
\]
Regarding the boundary values, we have 
$
\theta^{x_0}(x,y)\to 0
$ as $y\to 0$
for $x>x_0$ and $\theta^{x_0}(x,y)\to \pi$ as $y\to 0$ for $x<x_0$. 

For the open half-annulus with inner radius $S$ and outer radius $R$ and the straight part of its boundary, we use the notation \begin{equation}\label{eq:notationA}
A^+_{S,R}=\R^2_+ \cap (B_R(0)\setminus \overline{B_S(0)}), \qquad A^0_{S,R}=\{ x: S\le |x|\le R\}\subset\R
\end{equation}

With this notation, we have the following useful identity.
\begin{lemma}\label{l:lem1}
For
every function $V\in H^1_{\mathrm{loc}}(\RR^2_+)$ such that its trace on $A_{S,R}^0$ has values in
the interval $[-1,1]$,
we have the identity
\begin{align*}
 \int_{A_{S,R}^+} \frac{1}{2} |\nabla V|^2 \, dxdy = \frac{2}{\pi}\log
\frac{R}{S}&-\frac{2}{\pi}\int_{A_{S,R}^0} \frac{|V-\frac{2}{\pi}\theta^0+1|}{|x|}\,
dx \\ & +\int_{A_{S,R}^+} \frac{1}{2} |\nabla (V-\frac{2}{\pi}\,\theta^0)|^2 \, dxdy.
\end{align*}
\end{lemma}
\begin{proof}
We compute the integral in polar coordinates, with $r$ denoting the radius and $\theta$ the angle. Note that 
\[
|\nabla V|^2 = V_r^2 + \frac1{r^2}V_\theta^2.
\]
As $\theta^0_r=0$, $\theta^0_\theta=1$, we see that 
\[
|\nabla(V-\frac2\pi\theta^0)|^2 =  \frac1{r^2} \left( V_\theta-\frac2\pi\right)^2 + V_r^2.
\]
Writing 
\[V_\theta^2=\left( V_\theta-\frac2\pi+\frac2\pi\right)^2= \left( V_\theta-\frac2\pi\right)^2 -\frac4{\pi^2} +\frac4\pi V_\theta,
\]
we obtain
\begin{align*}
\frac12|\nabla V|^2 &= \frac12 V_r^2 + \frac1{2r^2}\left( V_\theta-\frac2\pi\right)^2 + \frac2{\pi r^2} V_\theta - \frac2{\pi^2r^2} \\
&=\frac12|\nabla(V-\frac2\pi \theta^0)|^2 + \frac2{\pi r^2} V_\theta - \frac2{\pi^2r^2}
\end{align*}

Integrating over $A^+_{S,R}$, we see
\begin{align*}
\int_{A_{S,R}^+} \frac12|\nabla V|^2\, dx dy&= \int_{A_{S,R}^+} \frac12|\nabla(V-\frac2\pi \theta^0)|^2 \, dx dy+ 
\frac2\pi \int_S^R \frac{V(re^{i\pi})-V(re^{i0})}{r}\, dr\\  \ &\qquad-\frac2\pi \log \frac RS \\
&= \int_{A_{S,R}^+} \frac12|\nabla(V-\frac2\pi \theta^0)|^2 \, dxdy \\ &\qquad -\frac2\pi \int_S^R \frac{(1-V(-r,0))+(1+V(r,0))}{r}\, dr +\frac2\pi \log \frac RS.
\end{align*}
As $(\frac2\pi\theta^0(x,0)-1)$ takes the value  $1$ for $x<0$ and the value $-1$ for $x>0$, 
we can use $-1\le V\le 1$ to see that
\[
|V(-r,0)-\frac2\pi \theta^0(-r,0)+1| =|V(-r,0)-1|=1-V(-r,0)\]
 and  \[
 |V(r,0)-\frac2\pi \theta^0(r,0)+1|=|V(r,0)+1| = 1+V(r,0).\]
We obtain the claim of the lemma.
\end{proof}

We now consider the so-called \emph{layer solution} and study its properties.
\begin{theorem}\label{thm:layerprops}
There exists a solution $U$ of
\begin{equation}\label{eq:layerPDE}
\left\lbrace
\begin{aligned}
&\Delta U =0 \qquad \,\,\text{ in } \R^2_+\\
&\frac{\de U }{\de \nu} = f(U) \quad\text{on }\partial\R^2_+=\R\times\{0\}
\end{aligned}
\right.
\end{equation}
with the following properties:
\begin{enumerate}[(i)]
\item $U_x>0$ and  $-1<U<1$ in $\overline{\R^2_+}$. In addition, $U(x,0)\to \pm 1$ as $x\to \pm\infty$.
\item $\int_\R G(U(x,0))dx<\infty$
\item With the additional constraint $U(0,0)=0$, the solution is unique. 
All other solutions of \eqref{eq:layerPDE} that satisfy (i) and (ii) are translations of $U$ in the $x$-~direction.
\item $\int_{\R^2_+\setminus \ol B_1^+} |\nabla  U + \frac2\pi \nabla \theta^0|^2 \, dxdy  <\infty$.
\item Setting $U_\e(x,y)=U(\frac {x\vphantom{y}} \e, \frac {y} \e)$
we have, for every $s>0$ and $M>1$
\[
\int_{A^+_{s,Ms}} \left(|\nabla U_\e + \frac2\pi \nabla \theta^0|^2 + \frac1{s^2} |U_\e-(1-\frac2\pi \theta^0)|^2 \right)\, dxdy \to 0
\]
as $\e\to 0$.
\end{enumerate} 
\end{theorem}
Regarding property (v), note that $\theta^0$ is invariant under the rescaling $(x,y)\mapsto (\frac x \e, \frac y \e)$. 
Hence we can view this as a statement that the layer $U$ is (before rescaling) $H^1$ close to 
the function $1-\frac2\pi \theta^0$ when $|(x,y)|$ is large.
\begin{proof}[Proof of Theorem \ref{thm:layerprops}]
The statements (i), (ii), and (iii) were proved in \cite{CabreSola-Mor:2005a}. 

To show (iv), we apply Lemma~\ref{lem5p5}
to the layer solution $U$ and to
 $V=1-\frac2\pi \theta^0$, with $\xi$ a test function that is radial and satisfies $0\le \xi\le 1$, $\xi=1$ in $A_{S,R}$, $\xi(z)=0$ for $|z|<\frac S 2$ or $|z|>2R$, and 
 such that $|\nabla \xi|<\frac c S$ in $A_{\frac S 2, S}$  as well as  $|\nabla \xi|<\frac c R$ in $A_{R, 2R}$. Then, with $D$ 
 defined as in Lemma~\ref{lem5p5},
 \[
 \int_{A_{S,R}} |\nabla  U + \frac2\pi \nabla \theta^0|^2 \, dxdy  \le D \]
 and
 \[
 D\le CD^{1/2} \left(\int_{A_{S/2,2R}} |\nabla \xi|^2 \right)^{1/2} + \int_{\frac S 2<|x|<2R} \xi^2(f(U)+\frac 1x) (U-(1-\frac2\pi \theta^0))\, dx.
 \]
 Using the mean value theorem and $f(1-\frac2\pi \theta^0)=0$ we note that  \[f(U)=f(U)-f(1-\frac2\pi \theta^0) =f'(\eta) (U-(1-\frac2\pi \theta^0))\] for a suitable $\eta$. For $S$ large enough,
we must have $f'(\eta)<0$. Hence we have the estimate 
\[
D\le CD^{1/2} + \int_{\frac S 2<|x|<2R} \frac{|U(x,0)-(1-\frac2\pi \theta^0(x,0))| }{|x|} \, dx.
\]

From the decay estimates (see Theorem 1.6 and (6.19) in \cite{CabreSola-Mor:2005a}, keeping in mind that the roles of $x$ and $y$ are reversed between
that paper and ours), we see for $S$ large enough that $|U(x,0)-(1-\frac2\pi \theta^0(x,0))| \le \frac C{|x|}$. Hence
\[
D\le CD^\frac12 + C\int_{\frac S2}^\infty  \frac{1}{x^2} \, dx,
\]
which leads to $D\le C$ independently of $R$, and hence(iv) follows. Part (v) follows from (iv) using a Poincar\'e inequality, keeping in mind (ii).
\end{proof}

Next we prove a result 
on the nonexistence of nontrivial solutions of homoclinic type 
(meaning that $\lim_{x\to-\infty} U(x,0)=\lim_{x\to\infty} U(x,0)$)
 that is needed later on. It is somewhat analogous to the classification results for degree $0$ and
degree  $\pm 1$ 
 solutions of the Ginzburg-Landau equations in $\R^2$ that are due to Brezis-Merle-Rivi\`ere 
 \cite{BrezisMerleRiviere:1994a} and Mironescu \cite{Mironesc:1996a}, respectively.
 
\begin{theorem}\label{thm:uni}
Let $V$ be a half-plane solution with values in $[-1,1]$ such that $|V(x,0)|\to 1$ as $x\to \pm \infty$ . Then 
either $V$ is constant $\pm 1$ or a layer solution, i.e., there exists $b\in \R$ with $V(x,y)=U(\pm (x-b),y)$.
\end{theorem}

For the case of the Peierls-Nabarro nonlinearity $f(u)=\frac{1}{\pi}\sin(\pi u)$, the classification result of Toland \cite{Toland:1997} shows that the
theorem holds true, with the increasing layer solution given by $\frac2\pi \phi^1$, with
\begin{equation}\label{phia}
\phi^a(x,y)=\arctan \frac{x}{y+a}.
\end{equation}
In this case, all solutions other than constants and the layer solutions are periodic, and hence none of them satisfies the homoclinic conditions about limits at $\pm\infty$.

For general $f$, the
case
 with the limit at $- \infty$ different from the limit at $+\infty$ was treated by Cabr\'e and Sol\`a-Morales \cite[Proposition 6.1]{CabreSola-Mor:2005a}, who showed that the 
 solution must be a (possibly reflected) translate of the layer solution. We proceed now to show 
that in the case of equal limits in Theorem~\ref{thm:uni}, the function must be constant. 

For the proof, we will use several times the following maximum principle for superharmonic functions in a half-space taken from the paper by the first author and Sol\`a-Morales~\cite{CabreSola-Mor:2005a}. Its proof in our case $n=2$ is simple; see~\cite{CabreSola-Mor:2005a}.

\begin{lemma}[\cite{CabreSola-Mor:2005a}, Lemma 2.8]
\label{lemSM} 
Let $d:\R^{n-1}\to\R$ be a continuous bounded function. Let $v$ be a bounded function in $\R^n_+$ satisfying 
\begin{equation} \label{maxpr}
\begin{cases}
-\Delta v\ge 0&\text{ in } \R^n_+\\ 
\dfrac{\partial v}{\partial\nu}+d(x)v\ge 0&\text{ on }\partial\R^n_+ =\R^{n-1}.
\end{cases}
\end{equation}
Assume that there exist a set $H\subset\partial\R^n_+=\R^{n-1}$ $($possibly empty$)$
and a constant $a>0$ such that
$$
v(x,0)> 0 \quad \mbox{for } y\in H \qquad \mbox{and}\qquad
d(x)\ge a^{-1} \quad \mbox{for } x\not\in H .
$$

Then $v>0$ in $\overline{\R^n_+}$, unless $v\equiv 0$.
\end{lemma}

\begin{proof}[Proof of Theorem \ref{thm:uni}]
As mentioned above, the case
 with the limit at $- \infty$ different from the limit at $+\infty$ was treated by Cabr\'e and Sol\`a-Morales \cite[Proposition 6.1]{CabreSola-Mor:2005a}, who showed that the 
 solution must be a (possibly reflected) translate of the layer solution. We proceed now to treat the case of equal limits.
 
Thus, let $V$ be a solution and assume that $V(x,0)\to -1$ as $x\to \pm\infty$. We proceed in three steps. In the two first ones we use the  function $\phi^a$ in \eqref{phia} to construct some comparison functions that will be used to show decay estimates for $\nabla V$.  In the third step, these estimates will 
allow us to use a Poho\u{z}aev identity from which we will then deduce that $V$ is constant.

\smallskip

{\it Step 1.  Here we show that}
\begin{equation*}
|\nabla V(x,0)| \le \frac{C}{x^2+1} \qquad\text{for all } x\in\R.
\end{equation*}

\smallskip

To show this, we start differentiating $\phi^a$ to find the harmonic function
\[
\phi^a_x = \frac{y+a}{x^2+(y+a)^2}.
\]
We also define another harmonic function $\hat \phi^a$ by differentiating again as follows:
\[
\hat \phi^a:=\phi^a_{xy}= \frac{x^2-(y+a)^2}{(x^2+(y+a)^2)^2}.
\]
On $\R=\partial \R^2_+$, we have 
\[
\frac{\partial}{\partial \nu} \phi^a_x = -\hat \phi^a(x,0)= \frac{a^2-x^2}{(x^2+a^2)^2}
\]
and hence
\[
\frac{\partial}{\partial \nu} \phi^a_x  + \frac1a \phi^a_x = \frac{a^2-x^2}{(x^2+a^2)^2}+ \frac{1}{x^2+a^2}= \frac{2a^2}{(x^2+a^2)^2} \ge 0
\]
on $\R$. 

The function $\tilde V=V+1$ is harmonic and satisfies $\tilde V\ge 0$ and  $\frac{\partial \tilde V}{\partial \nu} =f(\tilde V-1)$ on $\R$. As $f'(-1)<0$, we 
can choose $a, A>0$ such that 
\[
\frac{\partial \tilde V}{\partial \nu} +\frac1a \tilde V \le 0 \quad\text{ for }|x|\ge A.
\]
With this fixed $a$, for any $t>0$, we can now consider the harmonic function 
\[
\tilde V_t = t\phi^a_x-\tilde V,
\]
which satisfies 
\[
\frac{\partial \tilde V_t}{\partial \nu}+ \frac1a \tilde V_t \ge 0 \quad\text{ for }|x|\ge A.
\]
Choosing $t$ large enough, we have $\tilde V_t>0$ for $|x|\le A$ (as $\phi^a_x>0$). Using the maximum principle of Lemma~\ref{lemSM}, 
we see that $\tilde V_t>0$ on $\R^2_+$. Hence 
\begin{equation}\label{eq:decay1}
0\le \tilde V<\frac{t(y+a)}{x^2+(y+a)^2}.
\end{equation}

On $\R$, we have that 
\begin{equation}\label{eq:vyb}
|V_y(x,0)|=|f(V)|\le C|V+1| = C\tilde V\le \frac{C}{x^2+a^2}.
\end{equation}
The derivative $V_x$ is harmonic and satisfies
\[
\frac{\partial V_x}{\partial \nu} = f'(V)V_x.
\]
As $f'(-1)<0$, we can choose $A>0$, $a>0$ large enough such that for $|x|\ge A$,
\[
f'(V)\phi^a_x \le -\frac1a \phi^a_x \le \frac{\partial}{\partial\nu} \phi^a_x.
\]
We can now use Lemma~\ref{lemSM}  (with $d(\cdot)=-f'(V(\cdot,0))$, which is positive 
for $|x|>A$),
 we find for a sufficiently large $M>0$ that
\[
M\phi^a_x \pm V_x \ge 0 \quad\text{in $\R^2_+$},
\]
and hence $|V_x(x,0)|\le M\frac{a}{x^2+a^2}$. Adjusting the constant, together with \eqref{eq:vyb}, we find that for all $x\in \R$,
\begin{equation}\label{eq:fugrabl}
|\nabla V(x,0)| \le \frac{C}{x^2+1}.
\end{equation}

\smallskip

{\it Step 2.  We now find bounds on $|\nabla V|$ in the upper half-plane. We show that}
\begin{equation*}
|\nabla V|\le \frac {C}{x^2+(y+1)^2}\quad \text{in $\R^2_+$}.
\end{equation*}

\smallskip

We start by considering the truncated cone $K_1=\{(x,y)\in \R^2_+: |x|<2(y+1)\}$.

For $(x,y)\in K_1$, which we will assume for the moment, 
we have $\overline{B_{y/2}(x,y)}\subset \R^2_+$. Consider the harmonic 
functions $(\tilde V_x,\tilde V_y)=(\de_1 \tilde V, \de_2 \tilde V)$. 

Using the mean value theorem and the divergence theorem with $\nu=(\nu^1,\nu^2)$ the outer unit normal to $B_{y/2}(x,y)$, for $i=1,2$ we have
\[
\de_i \tilde V(x,y)= \avint_{B_{y/2}} \de_i \tilde V \, dx = \frac{4}{\pi y^2} \int_{B_{y/2}} \de_i \tilde V \, dx = \frac{4}{\pi y^2} \int_{\de B_{y/2}} \tilde V \nu^i\, d\ell.
\]
As \eqref{eq:decay1} gives $|\tilde V|\le \frac{C}{y+a}\le \frac C y $, we obtain (with constants changing from use to use)
\[
|\de_i W(x,y)| \le \frac C{y^2}\frac1y 2\pi \frac y 2 \le \frac{C}{y^2}
\] 
For $y>1$ and recalling $|x|<2(y+1)$, we can thus bound
\begin{equation}\label{eq:nWinK}
|\nabla V|=|\nabla \tilde V|\le \frac{C}{y^2}\le \frac{C}{(y+1)^2}\le \frac{C}{x^2+(y+1)^2}.
\end{equation}
For $0<y\le 1$, we have $|x|<4$ and hence$|\nabla \tilde V|\le C$; adjusting the constant we see \eqref{eq:nWinK} holds in the entire set $K_1$.

We consider the remaining set $\R^2_+\setminus K_1$, which consists of two conical sectors. Recall the harmonic function $\hat \phi:=\hat \phi^1$ from above, where 
we have set $a=1$. On 
the set $|x|=2(y+1)$, we have 
\[
\hat \phi = \frac{3(y+1)^2}{(5(y+1)^2)^2} = \frac{3}{25}\frac 1{(y+1)^2}.
\]
Choosing a suitably large constant $M$, we thus have, using our estimate  \eqref{eq:nWinK},
\[
M\hat \phi \pm \de_i \tilde V \ge 0 \quad\text{on }\{ |x|=2(y+1), y>0\}.
\]
For $|x|>2$ and $y=0$ we have for some small $c>0$
\[
\hat \phi = \frac{x^2-1}{(x^2+1)^2} \ge \frac {c}{x^2+1},
\]
and hence by  \eqref{eq:fugrabl} we also have (choosing $M$ larger if necessary)
\[
M\hat \phi \pm \de_i \tilde V \ge 0 \quad\text{on }\{ y=0, |x|>2\}.
\]
We can now use the maximum principle (see, e.g., Lemma 2.1 in Berestycki et al.~\cite{Berestycki:1997aa}), together with 
\eqref{eq:nWinK}, to conclude the decay estimate
\begin{equation}\label{eq:ubq}
|\nabla V| = |\nabla \tilde V| \le 2 M\hat \phi \le \frac {C}{x^2+(y+1)^2}\quad \text{in $\R^2_+$}.
\end{equation}

\smallskip

{\it Step 3.  Conclusion, using the Poho\u{z}aev identity.}

\smallskip

We use the Poho\u{z}aev identity for harmonic functions, Lemma \ref{lem:rellich}, to obtain for any large $R>0$
\[
R\int_{\de B_R\cap \R^2_+} \left| \frac{\de V}{\de \nu}\right|^2 \, d\ell + \int_{-R}^R f(V) V_x x \, dx  = \frac R 2 \int_{\de B_R\cap \R^2_+}
|\nabla V|^2 \, d\ell.
\]
Now integrating by parts in 1D we see, using $G'=-f$,
\[
\int_{-R}^R G(V)1 \, dx = R(G(V(R,0)+G(V(-R,0)) + \int_{-R}^R f(V) V_x  x\, dx,
\]
hence we see 
\begin{align*}
\int_{-R}^R G(V) \, dx &= R\big( G(V(R,0)+G(V(-R,0) \big) + \frac R 2 \int_{\de B_R\cap \R^2_+}
|\nabla V|^2 \, d\ell \\
& \qquad - R\int_{\de B_R\cap \R^2_+} \left| \frac{\de V}{\de \nu}\right|^2 \, d\ell .
\end{align*}
However, $G(V(\pm R,0))\le C|\tilde V(R,0)|^2\le \frac C{R^4}$ by  \eqref{eq:decay1} and $|\nabla V|^2 \le \frac C{R^4}$ by \eqref{eq:ubq}. 
Thus,
\[
\int_{-R}^R G(V) \, dx \le \frac {CR}{R^4} + \frac{CR^2}{R^4} \to 0 \quad\text{as }R\to \infty,
\]
hence $G(V)=0$ everywhere, and $V$ is constant.
\end{proof}

We now define a useful quantity, related to a Ginzburg-Landau quantity studied by Bethuel-Brezis-H\'elein \cite{BethuelBrezisHelein:1994a}: 

\begin{definition}
For $0<\e<\rho$, we set (with $U_\e$ as in (v) of Theorem~\ref{thm:layerprops}, normalized by setting $U_\e(0,0)=0$) 
\begin{equation}
I(\e,\rho) := \min \left\{ \frac12\int_{B_\rho^+} |\nabla  \Psi|^2 \,dxdy + \frac1\e\int_{-\rho}^\rho G( \Psi) \,dx : \Psi = U_\e \text{ on }\partial B_\rho^+\cap \R^2_+ \right\}.
\end{equation}
\end{definition}
Studying this quantity, we find a behavior similar to the result in Ginzburg-Landau theory.
\begin{proposition}\label{prop:i_er}
The quantity $I(\e,\rho)$ has the following properties:
\begin{enumerate}[(i)]
\item $I(\e, \rho)=I(1,\frac\rho\e)$.
\item $I(1,R)$ increases as $R>1$ increases.
\item With $U$ denoting the layer solution, we have 
\begin{equation}
I(1,R)=  \frac12\int_{B_R^+} |\nabla  U|^2 \,dxdy + \int_{-R}^R G( U) \,dx = \frac2\pi \log R + C_f +o(1)
\end{equation}
 as $R\to \infty$,
where $C_f$ is a constant depending only on the choice of nonlinearity $f$.
\end{enumerate}
\end{proposition}
\begin{remark}
We will not use this, but it may be interesting to notice that the same results hold (with the same value for $C_f$) for
a variant where the boundary condition is fixed instead of depending on $\e$:
\[
\hat I(\e,\rho) = \min \left\{ \frac12\int_{B_\rho^+} |\nabla  \Psi|^2 \,dxdy + \frac1\e\int_{I_\rho} G( \Psi) \,dx : \Psi = 1-\frac2\pi\theta^0 \text{ on }\partial B_\rho^+\cap \R^2_+ \right\}.
\]
For a special case, this was shown in \cite[Lemma 4.14]{IK_jac}.
\end{remark}
\begin{proof}[Proof of Proposition~\ref{prop:i_er}]
The first part of Proposition~\ref{prop:i_er} is obvious by rescaling. The first equality in (iii), and as a consequence (ii), follows from Lemma 3.1 in \cite{CabreSola-Mor:2005a}, which gives that 
the layer solution is minimizing with respect to its own boundary conditions. 

To prove the second equality in (iii), we write 
\[
I(1,R)=\int_{B_1^+} \frac12|\nabla U|^2 \, dx + \int_{I_R} G(U)\, d\ell+ \int_{A_{1,R}} \frac12|\nabla U|^2 \, dx.
\]
Using Lemma~\ref{l:lem1}, we find 
\[
\int_{A_{1,R}} \frac12|\nabla U|^2 \, dxdy=\frac2\pi \log R -\frac2\pi \int_{A^0_{1,R}} \frac{|U+\frac2\pi \theta^0-1|}{|x|}\,dx + \int_{A_{1,R}}  \frac12|\nabla U + \frac2\pi \theta^0|^2 \, dxdy
\]
Using Theorem~\ref{thm:layerprops} and the decay estimates from \cite{CabreSola-Mor:2005a}, we can let $R\to\infty$ and find that 
\[
I(1,R)-\frac2\pi \log R 
\]
converges as $R\to\infty$. The limit $C_f$ depends only on the layer solution, and hence only on the choice of $f$.
\end{proof}
\begin{remark}
For $f(u)=f^a(u)=\frac1{\pi a} \sin(\pi u)$ we can compute explicitly that 
\[
C_{f^a}=\frac2\pi \left( 1-\log a -a \log 2\right);
\]
 see
Theorem~\ref{thm:Cf} in the appendix.
\end{remark}

\section{Renormalized energy for boundary values in $\Ss^0$}\label{sec:renen}

In this section we collect some results on harmonic functions
that map the boundary of a two-dimensional domain into $\Ss^0=\{\pm 1\}$. 
The results are analogous to those for $\Ss^1$-valued harmonic 
maps found in \cite{BethuelBrezisHelein:1994a}. 
We also introduce some notation and certain conformal maps that will be used throughout the article. For conciseness, and since this is all we need subsequently, we only consider the 
case of two jumps. The general case, which is slightly more involved, can be treated with some modifications.

We generally assume that $p\neq q\in\dOm$ where $\Omega\subset\R^2$ is simply connected and, throughout this section, also smooth.
For small $\rho>0$, we set  $$\Omega_{\rho}=\Omega \setminus (\overline{B}_{\rho}(p)\cup
\overline{B}_{\rho}(q)).$$ We also let $\varphi:\Omega\to{\R^2_+}:=\{x+iy \, : \, x\in\R, y> 0\}$ 
 be a conformal map with inverse $\psi$. By the Kellogg-Warschawski theorem, $\varphi$ extends to a smooth
 function $\ol\Omega \to \overline{\R^2_+}=\{x+iy \, : \, x\in\R, y\ge 0\}$ whose inverse is the corresponding extension of $\psi$.

We now introduce the renormalized energy.
\begin{definition}\label{defi:renen}
The \emph{renormalized energy} $W$ is defined in $(\dOm\times\dOm)\setminus \{ (p,p):p\in \dOm\}$ by
$$ W(p,q)=W_{\Omega}(p,q):=\frac{2}{\pi}\log \frac{|\varphi(p)-\varphi(q)|^2}{|\varphi'(p)||\varphi'(q)|}
$$
\end{definition}
\begin{remark}\label{rem:crossr} 
Using the invariance of the cross-ratio under M\"obius transformations, it follows immediately that $W$ does not 
depend on the choice of the conformal map
 $\varphi$. We can thus choose $\varphi$ such that $\varphi(p)=0$ and $\varphi(q)=1$ 
and such that a point $p_\infty\in\dOm$ on the oriented
segment of $\dOm$ between $q$ and $p$ is mapped to $\infty$.
This determines $\varphi$ up to the choice of the point $p_\infty$. 

Furthermore, the renormalized energy can be computed without reference to conformal maps; see Proposition~\ref{propo:altren} below.
\end{remark}
\begin{example}
For $\Omega=B_1(0)$, a conformal map $\phi:\Omega\to \R^2_+$ is given by  $\varphi(z)=i\frac{1+z}{1-z}$.
After a short computation, we obtain that $\frac{|\varphi(p)-\varphi(q)|^2}{|\varphi'(p)||\varphi'(q)|}=|p-q|^2$.
Setting   $p=e^{i\sigma_1}$ and $q=e^{i\sigma_2}$ we have $|p-q|^2=2-2\cos(\sigma_1-\sigma_2)$, so
 \[
W_\Omega(p,q) = \frac4\pi \log|p-q|=\frac2\pi \log(2(1-\cos(\sigma_1-\sigma_2))).
\]
Hence, this renormalized energy is maximal for $p=-q$. From 
the expression in terms of $\sigma_1$ and $\sigma_2$, we easily see that it has no other critical points. 
\end{example}

The following theorem justifies calling $W$ the renormalized energy. Recall from \eqref{eq:chif} and \eqref{hext}
that $u_0^{p,q}$ is the 
harmonic function that only takes the  boundary values $-1$ and $1$, 
jumping from $1$ to $-1$ at $p$ and from $-1$ to $1$
at $q$ when the boundary is traversed counterclockwise.

\begin{theorem}\label{thm:calcW1}
The Dirichlet energy of $u_0=u_0^{p,q}$ has the following expansion as $\rho\to 0$:
\[ \frac12\int_{\Omega_{\rho}} |\nabla u_0|^2 \, dx= \frac 4\pi  \log \frac{1}{\rho} +
W_{\Omega}(p,q) + O(\rho). \]
The constant implicit in the $O(\rho)$ notation depends only on the domain $\Omega$.
\end{theorem}
\begin{proof}
We set $U_0=u_0\circ \psi$, and hence $U_0:\R^2_+\to \R$. By the conformal invariance of the Dirichlet integral, we find that 
\[
\frac12\int_{\Omega_\rho} |\nabla u_0|^2 \, dxdy = \frac12\int_{\varphi(\Omega_\rho)} |\nabla U_0|^2 \, dxdy.
\]

We note that $U_0$ is a bounded harmonic function on $\R^2_+$ whose boundary values satisfy
 $U_0(x,0)=-1$ for $0<x<1$ and $U_0(x,0)=1$ for $|x-\frac12|>1$.
By uniqueness, it follows that we must have 
\begin{equation}\label{eq:p7a}
-U_0=\frac2\pi(\theta^1-\theta^0)-1 = \frac2\pi \theta^1-2+(1-\frac2\pi \theta^0)
\end{equation}
In the following we take two positive numbers $r_0,r_1$ with $r_0+r_1<1$ and will assume $R>1$.
In the domain $H^R_{r_0,r_1}= B^+_R(0) \setminus (\overline{B_{r_0}(0)}\cup \overline{B_{r_1}(1)}$, we compute,
integrating the cross term arising from the decomposition \eqref{eq:p7a} by parts,
\begin{align*}
\frac12\int_{H^R_{r_0,r_1}} |\nabla U_0|^2\, dxdy &= \frac2{\pi^2}\int_{H^R_{r_0,r_1} }
\left( {|\nabla \theta^1|^2} 
+{|\nabla \theta^0|^2}\right) \, dxdy \\
&\quad +\int_{\partial H^R_{r_0,r_1}} \left(\frac2\pi \theta^1-2\right)\frac{\partial}{\partial\nu}\left(1-\frac2\pi \theta^0\right) \,d\ell.   
\end{align*}
Note that $|\nabla \theta^0(re^{i\theta})|=\frac1r$ and hence, as $r_1\to 0$, $r_2\to 0$ and $R\to \infty$, we have
\[
\int_{{B^+_{r_1}(1)}} |\nabla \theta^0|^2 \, dxdy = O(r_1^2),\quad \int_{{B^+_{r_0}(0)}} |\nabla \theta^1|^2 \, dxdy = O(r_0^2)
\]
and 
\[
\int_{B^+_{R}(0) \setminus \overline{B_{R-1}(1)}} |\nabla \theta^1|^2 \, dxdy \le \pi \int_{R-1}^R \frac1r \, dr =O(\frac1R).
\]
Note that this leads to
\[
\int_{H^R_{r_0,r_1}} |\nabla \theta^0|^2 \, dxdy  = \frac2\pi \log \frac{R}{r_0}+ O(r_1^2) 
\]
and
\[
\int_{H^R_{r_0,r_1}} |\nabla \theta^1|^2 \, dx dy = \frac2\pi \log \frac{R}{r_1}+ O(r_0^2)+O(\frac1R).
\]

For the boundary term, we note that $\frac{\partial\theta^0}{\partial\nu}=0$ on $\partial B_{r_0}(0)$ and $\partial B_R(0)$, while 
$\frac{\partial\theta^0}{\partial\nu}(x,0)=-\frac1x$ for $x>1$. As $\frac2\pi \theta^1(x,0)$ is equal to $2$ for $x<1$ and $0$ for $x>1$, we find
\begin{align*}
\int_{\partial H^R_{r_0,r_1}} \left(\frac2\pi \theta^1-2\right)\frac{\partial}{\partial\nu}\left(1-\frac2\pi \theta^0\right) \,d\ell &=
-\frac4\pi \int_{1+r_1}^R \frac1x\, dx + O(r_1)\\ &=-\frac4\pi \log R + O(r_1).
\end{align*} 
Combining everything, we obtain 
\begin{align*}
\frac12\int_{H^R_{r_0,r_1}} |\nabla U_0|^2\, dxdy &= \frac2\pi \log\frac{R}{r_0} + \frac2\pi \log\frac{R}{r_1}  +O(r_0)+O(r_1)-\frac4\pi \log R+O(\frac1R)
\\ &=  \frac2\pi \log \frac{1}{r_0r_1} + O(r_0)+O(r_1)+O(\frac1R).
\end{align*}
Letting $ R\to \infty$ and writing $H_{r_1,r_2}=\R^2_+\setminus{\ol{B_{r_0}(0)}\cup \ol{B_{r_1}(1)}}$, we see that
\begin{equation}\label{eq:dirintO}
\frac12\int_{H_{r_1,r_2}}|\nabla U_0|^2\, dx dy= \frac2\pi \log \frac{1}{r_0r_1} + O(r_0)+O(r_1).
\end{equation}
From the smoothness of $\varphi$, it follows that $\phi(\Omega_\rho)$ is
for small $\rho$ close to  $\R^2_+$ with two suitable half-balls removed; more precisely
there exists $c>0$ with 
\[
H_{|\phi'(p)|\rho+c\rho^2,|\phi'(q)|\rho+c\rho^2} \subset \phi(\Omega_\rho) \subset H_{|\phi'(p)|\rho+c\rho^2,|\phi'(q)|\rho+c\rho^2}.
\]
We use \eqref{eq:dirintO} for $r_0=|\varphi'(p)|\rho\pm c\rho^2$ and $r_1=|\varphi'(q)|\rho\pm c\rho^2$.

Using the conformal invariance of the Dirichlet integral this leads to
\begin{multline*}
\frac12 \int_{\Omega_\rho} |\nabla u_0|^2\, dx dy \ge 
\frac2\pi \log \frac{1}{(|\varphi'(p)|\rho+ c\rho^2 )(|\varphi'(q)|\rho+ c\rho^2)} \\+ O(|\varphi'(p)|\rho+ c\rho^2)+O(|\varphi'(q)|\rho+ c\rho^2)
\end{multline*}
and the reverse inequality 
\begin{multline*}
\frac12 \int_{\Omega_\rho} |\nabla u_0|^2\, dx  
 \le \frac2\pi \log \textstyle\frac{1}{(|\varphi'(p)|\rho- c\rho^2 )(|\varphi'(q)|\rho- c\rho^2)} \\+ O(|\varphi'(p)|\rho- c\rho^2)+O(|\varphi'(q)|\rho- c\rho^2)
\end{multline*}
Since 
\[\log \frac{1}{|\varphi'(p)|\rho\pm c\rho^2} = \log \frac1{\rho|\varphi'(p)|}+ \log \frac{1}{1\pm \frac{c\rho}{|\varphi'(p)|}}
= \log \frac1{|\varphi'(p)|}+ \log \frac 1\rho +O(\rho),
\]
we can thus compute
\begin{align*}
\frac12\int_{\Omega_\rho} |\nabla u_0|^2 \, dx = \frac4\pi \log \frac1{\rho} + \frac2\pi 
\log \frac1{|\varphi'(p)||\varphi'(q)|} + O(\rho).
\end{align*}
As $\varphi(q)-\varphi(p)=1$, we obtain the claim. The constants in the $O$ notation depend only on $\varphi$, hence only on $\Omega$. 
\end{proof}
The following alternate characterisations of $W$ hold.
\begin{proposition}\label{propo:altren}
We have 
\begin{align*}
W_{\Omega}(p,q)&=-\frac{2}{\pi} \log \frac{\partial ^2 G^D}{\partial \nu_p
\partial \nu_q}(p,q) \\
&=2(R^N(p,q)+R^N(q,q)-2 G^N(p,q))\\
&= {\frac{4}{\pi} \log |p-q| + 2 (R^N(p,p)+R^N(q,q)-2R^N(p,q))}\, ,
\end{align*}
for $p,q\in \partial \Omega$ with $p\neq q$ and with $G^D(p,q)$ the Dirichlet Green's function of
$\Omega$, $G^N(p,q)$ the Neumann Green's function of $\Omega$, and $R^N(p,q)$ the regular
part of the last one.
\end{proposition}
\begin{proof}
Let $\varphi:\Omega \to \R_+^2$ be a conformal map and $p,q\in \Omega$. Then
$$ G^D(p,q) = -\frac{1}{2\pi} \left\{ \log|\varphi(p)-\varphi(q)| -\log
|\varphi(p)-\overline{\varphi(q)}|\right\}\, . $$ Now, for $p\in \Omega$ and $q\in \partial \Omega$
we obtain the Poisson kernel as
$$K^D(p,q)=\frac{\partial G^D}{\partial \nu_q}(p,q) = \frac{1}{\pi}
\frac{1}{|\varphi(p)-\varphi(q)|^2}\, , $$ where $\nu_q$ denotes the outer normal unit vector at
$q$. For $p,q\in \partial \Omega$ the outer normal derivative in $p$ gives the first equality, that
is,
$$ \frac{\partial^2G^D}{\partial \nu_p \partial \nu_q}(p,q)=\frac{1}{\pi}
\frac{|\varphi'(p)||\varphi'(q)|}{|\varphi(p)-\varphi(q)|^2}\, .$$

For the Neumann conditions, the Green's function for $p\in \Omega$ and $q\in \partial \Omega$ is
$$ G^N(p,q)=-\frac{1}{\pi} \log |p-q| + R^N(p,q)\, ,$$
where $ R^N(p,q) = -\frac{1}{\pi} \log \frac{|\varphi(p)-\varphi(q)|}{|p-q|}$ is the regular part.
Let us note that $R^N(q,q)=-\frac{1}{\pi}\log |\varphi'(q)|$, taking $p\to q\in \partial \Omega$.
Hence, with this notation we deduce directly the last two equalities. 
\end{proof}

\section{Minimization in a closed neighborhood and energy upper bounds}\label{sec:conmin}
In this section we begin our analysis of the main problem by looking at minimizers of the energy
restricted to a suitable subset of $H^1(\Omega)$. 

Recall from \eqref{eq:chif} that we defined $\chi^{p,q}$ as a function that is $1$ on the  counterclockwise oriented segment between $p\in\partial\Omega$ 
and $q\in\dOm$, and $-1$ on the oriented segment between $q$ and $p$.

In this section we consider functions that are close to $\chi^{p,q}$ in $L^2(\dOm)$ and solve a PDE related
 to \eqref{eq:pde}. We will do so 
by solving a  minimization problem associated to \eqref{eq:pde} with a suitable restriction.

We first approximate the discontinuous function $\chi^{p,q}: \dOm\to \{-1,1\}$ by smooth functions.
\begin{lemma}\label{lem42}
For $\e>0$ small enough, there exist smooth functions 
 $\chi^{p,q}_\e:\dOm\to [-1,1]$ with $\chi^{p,q}_\e(z)=\chi^{p,q}(z)$ for $z\in \dOm\setminus (B_\e(p)\cup B_\e(q))$ and
$\|\chi^{p,q}_\e\|_{C^k(\dOm)}\le \frac C{\e^k}$ for $k=1,2$.
\end{lemma}
\begin{proof}
It is straightforward to construct functions with these properties by rescaling a smooth approximation to 
 the signed distance function to $p$ and $q$. 
\end{proof}

For $a>0$, we consider  the sets
\[
\calC_a^\e = \{ u\in H^1(\Omega) : \|u-\chi^{p,q}_\e \|_{L^2(\dOm)}^2 \le a^2 \}.
\]
By the Rellich-Kondrashov theorem, the sets $\calC_a^\e$ are closed under weak $H^1(\Omega)$ convergence. The direct method
in the calculus of variations readily gives the existence of a minimizer $u_\e$ of $E_\e$ among $u\in \calC_a^\e$.
Any such minimizer clearly satisfies that $-1\le u_\e\le 1$. 
We will estimate the energy of a minimizer, and
start by proving a near-optimal upper bound by showing the existence of suitable comparison functions.
Proving the corresponding lower bound will be significantly more involved.

\begin{proposition}\label{prop:ub}
For every $a>0$ and  $p\neq q\in\dOm$, for small $\e>0$ there exist functions $v_\e\in \calC^\e_a $ such that $-1\le v_\e\le 1$, $v_\e\to \chi^{p,q}$ in $L^2(\dOm)$, and 
\[
E_\e(v_\e) \le \frac4\pi \log \frac1\e + W_\Omega(p,q) + 2 C_f + o_\e(1).
\]
\end{proposition}
\begin{proof}
The idea is to combine the discontinuous $u_0$, the extension of $\chi^{p,q}$ to $\Omega$ found in \eqref{hext}, with the layer solution. We will 
 first define a function $V_\e$ in the half-plane and then set $v_\e=V_\e \circ \psi$, where $\psi$ is the 
inverse of the conformal map $\varphi$ chosen as in Remark~\ref{rem:crossr}.
Then
\[
E_\e(v_\e) = \frac12\int_{\R^2_+} |\nabla V_\e|^2 \, dxdy + \frac1\e \int_{\R} |\psi'(x)|G(V_\e) \, dx
\]
With $h_0=|\psi'(0)|=\frac1{|\varphi'(p)|}$ and $h_1=|\psi'(1)|=\frac1{|\varphi'(q)|}$ we now define the function $\hat U_\e$ by
\[
\hat U_\e(z) = \begin{cases} U(\frac{h_0 z}{\e}) & z\in B^+_{\rho/h_0}(0) \\
U(-\frac{h_1 (z-1)}{\e}) & z\in B^+_{\rho/h_1}(1),
\end{cases}
\]
where $U(z)$ is the increasing layer solution found in Theorem~\ref{thm:layerprops}.
With a cutoff function $\xi$ that is $0$ in $B^+_{\rho/h_0}(0)\cup B^+_{\rho/h_1}(1)$, $1$ in 
$\R^2_+\setminus(\ol{B_{2\rho/h_0}(0)}\cup \ol{B_{2\rho/h_1}(1)})$ and satisfies $|\nabla \xi|\le \frac C \rho$, we 
can now set (with $u_0=U_0\circ\psi$ as in the proof of Theorem~\ref{thm:calcW1})
\[
V_\e=\xi U_0 +(1-\xi)\hat U_\e.
\]

We compute the energy of $V_\e$. Near $0$, we have
\begin{multline*}
\frac12\int_{B_{\rho/h_0}^+(0)} |\nabla V_\e|^2 \, dxdy+ \frac1\e \int_{\R\cap B_{\rho/h_0}(0)} |\psi'(x)|G(V_\e) \, dx \\
= \frac12\int_{B_{\rho/h_0}^+(0)} |\nabla V_\e|^2 \, dxdy + \frac{h_0}\e \int_{\R\cap B_{\rho/h_0}(0)} G(V_\e) \, dx \\
\quad+  \frac1\e \int_{\R\cap B_{\rho/h_0}(0)} (|\psi'(x)|-|\psi'(0)|) G(V_\e) \, dx
\end{multline*}
Rescaling by $h_0/\e$ and using $\int_\R G(U) dx <\infty$ and $|\psi'(x)-\psi'(0)|\le C\rho$, we obtain
\begin{multline}\label{eq:ubnear0}
\frac12\int_{B_{\rho/h_0}^+(0)} |\nabla V_\e|^2 \, dxdy + \frac1\e \int_{\R\cap B_{\rho/h_0}(0)} |\psi'(x)|G(V_\e) \, dx \\
= \frac12\int_{B_{\rho/\e}^+(0)} |\nabla U|^2 \, dxdy + \int_{\R\cap B_{\rho/\e}(0)} G(U) \, dx +O(\rho) =I(\e,\rho)+O(\rho).
\end{multline}
The same argument applied to the region near $1$ yields
\begin{equation}\label{eq:ubnear1}
\frac12\int_{B_{\rho/h_1}^+(1)} |\nabla V_\e|^2 \, dxdy + \frac1\e \int_{\R\cap B_{\rho/h_1}(1)} |\psi'(x)|G(V_\e) \, dx  =I(\e,\rho)+O(\rho).
\end{equation}
By Theorem~\ref{thm:calcW1} (compare \eqref{eq:dirintO})
we find 
\begin{equation}\label{eq:tuouter}
\frac12\int_{\R^2_+\setminus(\ol{B_{\rho/h_0}(0)}\cup \ol{B_{\rho/h_1}(1)})} |\nabla U_0|^2dxdy = 
\frac4\pi \log \frac1\rho + W_\Omega(p,q)+O(\rho).
\end{equation}

Computing $|\nabla V_\e|^2-|\nabla U_0|^2 = \nabla (V_\e-U_0)\cdot \nabla (V_\e+U_0)$,
we note that
\[
\nabla (V_\e-U_0) = (\xi-1) \nabla(U_0-\hat U_\e) + \nabla \xi (U_0-\hat U_\e).
\]
hence
\begin{align*}
\int_{A_{\rho/h_0,2\rho/h_0}^+} |\nabla (V_\e-U_0) |^2\, dxdy &\le 2\int_{A_{\rho/h_0,2\rho/h_0}^+(0)} 
|\nabla (U_0-\hat U_\e)|^2 \, dxdy \\ &\qquad + \frac{c}{\rho^2} \int_{A_{\rho/h_0,2\rho/h_0}^+(0)} | U_0-\hat U_\e|^2 \, dx.
\end{align*}
Note that $U_0 =(1-\frac2\pi \theta^0)+ \frac2\pi \theta^1-2$ and 
$\frac2\pi \theta^1-2=O(1)=|\nabla( \frac2\pi \theta^1-2)|$.
By part (v) of Theorem~\ref{thm:layerprops},  $\hat U_\e \to (1-\frac2\pi \theta^0)$ in 
$H^1(A_{\rho/h_0,2\rho/h_0}^+)$  as $\e\to 0$ and 
\begin{equation}\label{eq:annd}
\int_{A_{\rho/h_0,2\rho/h_0}^+} |\nabla (V_\e- U_0) |^2\, dxdy =O(\rho^2).
\end{equation}
By the convergence in Proposition~\ref{prop:i_er} and $|\nabla U_0|\le \frac C \rho$,
\begin{equation}\label{eq:anns}
\int_{A_{\rho/h_0,2\rho/h_0}^+} |\nabla V_\e|^2+|\nabla U_0 |^2\, dxdy \le C.
\end{equation}
Combining \eqref{eq:annd} and \eqref{eq:anns},
we see using Cauchy-Schwarz that
\begin{equation}\label{eq:diffda}
\frac12\int_{A_{\rho/h_0,2\rho/h_0}^+} \left(|\nabla V_\e|^2-|\nabla U_0 |^2\right)\, dxdy  \le C\rho.
\end{equation}
For the boundary term, we note that for $\e$ small, 
\(
G(V_\e)\le G(\hat U_\e),
\)
hencefrom $\int_{\R} G(U)\, dx<\infty$ and rescaling we have 
\begin{equation}\label{eq:penouter}
\frac1\e\int_{A_{\rho/h_0,2\rho/h_0}^0} |\psi'(x)| G(\hat U_\e)\, dx \le (1+C\rho) \int_{A^0_{\rho/\e,{2\rho}/\e}} G(U)\, dx\to 0
\end{equation}
as $\e\to 0$ (and the corresponding result near $1$).

Combining \eqref{eq:ubnear0}, \eqref{eq:ubnear1}, \eqref{eq:tuouter}, \eqref{eq:diffda} and \eqref{eq:penouter}, 
we find
\[
\lim_{\e\to 0} \left(E_\e(v_\e)  - 2I(\e, \rho)- \frac4\pi \log \frac1\rho + W_\Omega(p,q)\right) =O(\rho).
\]
From our results on $I(\e,\rho)$ we now obtain the claim on the energy by letting $\rho\to 0$. From the construction,
it is clear that $-1\le v_\e\le 1$ and $v_\e\to \chi^{p,q}$ in $L^2(\dOm)$.

We also have $v_\e=\chi^{p,q}_\e$ outside $\partial\Omega \cap(B_{C\rho}(p)\cup B_{C\rho}(q))$.
Thus  $\|v_\e-\chi^{p,q}_\e\|^2_{L^2(\partial\Omega)}\le C\rho$, and hence for $\rho$ small enough we find
$v_\e\in \calC_a^\e$.  
\end{proof}

We now consider the Euler-Lagrange equations satisfied by minimizers of the energy in the 
set $\calC_a^\e$. We establish basic smoothness results and a bound on the Lagrange multiplier.
\begin{proposition}\label{pro:conmin}
Minimizers $u_\e$ of $E_\e$  in the set $\calC_a^\e$ satisfy the Euler-Lagrange equations
\begin{equation}\label{eq:elelam}
\begin{cases}
\Delta u_\e &=0\qquad\quad\qquad\qquad\qquad\text{ in $\dOm$}\\
\frac{\partial u_\e}{\partial \nu} &=\frac1\e f(u_\e)+ \lambda_\e (u_\e-\chi^{p,q}_\e) \, \text{ on $\dOm$}
\end{cases}\end{equation}
for some $\lambda_\e\in \R$ that satisfies
\begin{equation}\label{eq:lamu}
|\lambda_\e|\le C\sqrt\frac{|\log\e|}{\e}.
\end{equation}
In addition, all solutions of \eqref{eq:elelam} are $C^{2,\beta}$ smooth up to the boundary and satisfy 
\begin{equation}\label{eq:gradb}
|\nabla u_\e|\le \frac C\e\quad\text{in $\ol\Omega$}.
\end{equation}
\end{proposition}
\begin{proof}
The existence of $\lambda_\e$ and the equations \eqref{eq:elelam} follow from the Lagrange multiplier theorem. 

For the bound on the Lagrange multiplier, note that either 
$\lambda_\e=0$ or $\int_\dOm |u_\e-\chi^{p,q}_\e|^2\, d\ell =a^2$. In the latter case, we can use the function $v_\e$ from 
Proposition~\ref{prop:ub} and the energy bound
 $E_\e(u_\e)\le E_\e(v_\e)\le C|\log\e|$. We have $|v_\e-\chi^{p,q}_\e|\to 0$ in $L^2(\dOm)$. Testing $\Delta u_\e=0$ with 
 $(u_\e-v_\e)$, we find from \eqref{eq:elelam} that
 \[
 \lambda_\e \int_\dOm (u_\e-\chi^{p,q}_\e) (u_\e-v_\e)\, d\ell = \int_\Omega \nabla u_\e \cdot (\nabla u_\e -\nabla v_\e)\, dx -\frac1\e
 \int_\dOm f(u_\e)(u_\e-v_\e)\, d\ell.
 \]
We have \[\int_\dOm (u_\e-\chi^{p,q}_\e) (u_\e-v_\e)\, d\ell \to a^2\] and 
\[ \left| \int_\Omega \nabla u_\e \cdot (\nabla u_\e -\nabla v_\e)\, dx  \right| \le C |\log\e|.\]
For the penalty term, we estimate, using $\bigl(f(u)\bigr)^2\le CG(u)$  (this follows from the fact that $G$ has a differentiable square root) that
\begin{align*}
\left|\frac1\e
 \int_\dOm f(u_\e)(u_\e-v_\e)\, d\ell \right| & \le \frac C{\e^{1/2}}  \left(\int_\dOm \frac1\e \bigl(f(u_\e)\bigr)^2 \, d\ell\right)^{1/2} 
 \\ & \le \frac C{\e^{1/2}}  \left(\int_\dOm \frac1\e G(u_\e) \, d\ell\right)^{1/2}  \le \frac{C|\log\e|^{1/2}}{\e^{1/2}}.
\end{align*}
Overall, we obtain
\[
|\lambda_\e|(a^2-o(1)) \le C\left( \frac{|\log\e|^{1/2}}{\e^{1/2}} + |\log\e|\right),
\]
which proves \eqref{eq:lamu} for sufficiently small $\e$.

For the regularity and \eqref{eq:gradb}, we follow with some small modifications the proof of 
\cite[Lemma 2.3]{CabreSola-Mor:2005a}. Interior regularity and the bound 
$|\nabla u_\e |\le \frac C{\mathrm{dist}(\cdot, \dOm)}$
are classical. For boundary regularity and \eqref{eq:gradb}, by compactness of $\dOm$ it
is enough to show that there exists $R>0$ such that
for every $z\in \dOm$ there is $r\ge R\e$ such that $u_\e$ is regular and
the estimate 
holds in $\Omega \cap B_r(z)$. 

Let $\Phi_z:\R^2_+\to \Omega$ be a conformal map with $\Phi_p(0)=p$ and $|\Phi_p'(0)|=1$. 
Then $v_\e=u_\e\circ \Phi_z$ satisfies 
\begin{equation}
\begin{cases}\label{eq:inhalfsp}
\Delta v_\e &=0 \quad\quad\quad \text{ in } \R^2_+ \\
\frac{\partial v_\e}{\partial \nu} &=\frac{a(x)}\e f(v_\e)+ a(x)\lambda_\e (v_\e-\hat\chi^{p,q}_\e) \quad \text{ on }\R,
\end{cases}
\end{equation}
where $a(x)=|\Phi'(x)|$ and 
 $\hat \chi^{p,q}_\e = \chi^{p,q}_\e \circ \Phi_z$.
After rescaling by $\e$, we find that $V_\e(z)=v_\e(\e z)$ satisfies
\begin{equation}
\begin{cases}\label{eq:inhalfsp2}
\Delta V_\e &=0 \quad\quad\quad\text{ in } \R^2_+ \\
\frac{\partial V_\e}{\partial \nu} &={a(\e x)}f(V_\e)+ a(\e x)\e\lambda_\e (V_\e-X^{p,q}_\e)
\quad \text{ on }\R,
\end{cases}
\end{equation}
where $X^{p,q}_\e$ is a function with locally bounded first and second derivatives.
We need to show that in a ball of radius $R$, 
$|\nabla V_\e|\le C$, with $C$ and $R$ chosen independent of $\e$.

We now integrate in $y$-direction, setting 
\[
W_\e(x,y)=\int_0^y V_\e(x,t)\, dt.
\]
Then $(\Delta W_\e)_y =0$. Thus
\begin{equation}\label{eq:WEve}
\Delta W_\e(x,y)=\Delta W_\e(x,0)= V_\e,y(x,0)=-a(\e x) f(V_\e) -a(\e x)\e\lambda_\e (V_\e-X^{p,q}_\e).
\end{equation}
We find $\|\Delta W_\e\|_{L^\infty(B^+_{2R})} \le C_R$, where $C_R$ can be chosen independent of $\e$.
From $W_\e=0$ on $\R$ and elliptic regularity, we see
$\|W_\e\|_{W^{2,p}(B^+_{2R})} \le C_R$ for $2<p<\infty$ and hence $W_\e\in C^{1,\beta}$ for $0<\beta<1$.
This leads to $\|V_\e\|_{C^\beta(B^+_{2R})}\le C_R$ and, plugging this into \eqref{eq:WEve}, we see
using the smoothness of $f$ and $\e\lambda_\e\to 0$ that for $\e<\e_0$,
\[
\|\Delta W_\e\|_{C^\beta(B^+_{2R})} \le C_R.
\]
Schauder estimates then yield $\|W_\e\|_{C^{2,\beta}(B^+_R} \le C_R$, which 
gives $\|V_\e\|_{C^{2,\beta}(B^+_R} \le C_R$. In particular, we have a uniform gradient bound on $V_\e$.
Plugging again into \eqref{eq:WEve}, we see $\|\Delta W_\e\|_{C^\beta(B^+_{R})} \le C_R$, 
and Schauder estimates then show $W_\e\in C^{3,\beta}$, and thus $V_\e\in C^{2,\beta}$. 
\end{proof}

\section{Poho\u{z}aev balls and covering arguments}
 \label{sec:poho}
In this section we work towards a lower bound complementing the results of Proposition~\ref{prop:ub},
by using a Poho\u{z}aev type identity and covering estimates near approximate transitions.
Our results extend ideas in \cite{Kurzke:2006b}, which in turn 
are modeled on results for the Ginzburg-Landau equation 
in \cite{BethuelBrezisHelein:1994a,Struwe:1994a}. Compared to 
\cite{Kurzke:2006b}, we additionally have to deal with the asymptotically
small term coming from the Lagrange multiplier in \eqref{eq:elelam} 
if it is active. We will incorporate this as an extra term $h_\e(x)$ in the Neumann boundary condition.

In the following, we assume that for $0<\e\ll 1$, 
 $(u_\e)$ are $C^2$ solutions of 
\begin{equation}
\begin{cases}
\Delta u_\e & =0 \quad \text{in $\Omega$}\\
\frac{\partial u_\e}{\partial \nu} &= \frac1\e f(u_\e)+ h_\e(x) \quad\text{on $\dOm$}
\label{eq:bc1}
\end{cases}
\end{equation}
and that there exists $\beta_0\in(\frac12,1)$ with 
\begin{equation}\label{eq:eb0heto0}
\e^{\beta_0} h_\e\to 0 \quad\text{in}\quad  L^\infty(\dOm).
\end{equation}
 We also assume 
 \[|\nabla u_\e|\le \frac C \e\] 
 and 
 \[E_\e(u_\e)\le K|\log\e|\] for some $K>0$.
By Proposition~\ref{pro:conmin}, these assumptions are satisfied for  minimizers in the closed neighborhood $\calC_a^\e$.
Note that for sufficiently small $\e$, $K$ can be any number with $K>\frac4\pi$.
Our assumptions on $f$ and $G$ are as in Section~\ref{sec:layer}; in particular, we will now use the assumptions on $G$ to
 choose a number $t_*\in (0,1)$ such that $G$ is convex in each of the two connected components of 
 $\{t\in \R: t_*\le |t|\le 1\}$.

For harmonic functions on 2d domains, the Poho\u{z}aev  (or Rellich) identity takes the following form.
The domain $D$ in the statement will later be an appropriate subset of $\Omega$.
\begin{lemma}\label{lem:rellich}
For any Lipschitz domain $D$ and any harmonic function $u\in C^2(D)\cap C^1(\ol D)$, 
there holds
\begin{equation}\label{eq:rellich}
\int_{\partial D} \frac{\partial u}{\partial \nu} (z\cdot \nabla u) \, d\ell
= \frac12\int_{\partial D} (z\cdot \nu) |\nabla u|^2\, d\ell,
\end{equation}
where $z=(x,y)$.
\end{lemma}
\begin{proof}
We can verify by direct calculation that
$\nabla u \cdot \nabla(z\cdot \nabla u) = \frac12\nabla \cdot (z|\nabla u|^2)$, and hence the divergence theorem gives \eqref{eq:rellich}.
\end{proof}
As in \cite{Kurzke:2006b}, one can easily deduce the equivalence of the $L^2$ norms of the tangential and normal derivatives 
of harmonic functions on star-shaped domains, stated next.
\begin{lemma}\label{lem:comparenutau}
Let $D\subset \R^2$ be a Lipschitz domain with the property that 
for some $z^*\in D$ and some $k>0$, $(z-z^*)\cdot \nu > k|z-z^*|$ for all 
$z\in \partial D$ (such domains are strictly starshaped). Then, 
there are constants $0<c_1<c_2$ depending only on the shape of $D$
(more precisely, depending on $k$ and on $\max_{z\in \partial D}|z-z^*| / \min_{z\in \partial D}|z-z^*|$)
such that any harmonic function $u\in C^2(D)\cap C^1(\overline{D})$ satisfies
\[
c_1 \int_{\partial D} \left|\frac{\partial u }{\partial\tau} \right|^2 \,d\ell
\le \int_{\partial D} \left|\frac{\partial u }{\partial \nu} \right|^2\,d\ell
\le c_2\int_{\partial D} \left|\frac{\partial u }{\partial \tau} \right|^2\,d\ell,
\]
where $\tau=\nu^\perp$ denotes a unit tangent to $\partial D$.
\end{lemma}
\begin{proof}
From \eqref{eq:rellich}, we have after a change of coordinates
\[
\int_{\partial D} \frac{\partial u}{\partial \nu} ((z-z^*)\cdot \nabla u) \, d\ell
= \frac12\int_{\partial D} ((z-z^*)\cdot \nu) |\nabla u|^2\, d\ell
\]
In particular, writing $\nabla u = \frac{\partial u}{\partial \nu}\nu +\frac{\partial u}{\partial \tau}\tau$,  
\begin{align*}
\int_{\partial D} \left|\frac{\partial u}{\partial \nu}\right|^2 & ((z-z^*)\cdot \nu) \, d\ell
+ \int_{\partial D} \frac{\partial u}{\partial \nu}\frac{\partial u}{\partial \tau} ((z-z^*)\cdot \tau) \, d\ell
\\&
=\frac12\int_{\partial D} \left(\left|\frac{\partial u}{\partial \nu}\right|^2+\left|\frac{\partial u }{\partial \tau} \right|^2
\right)(z-z^*)\cdot\nu \,d\ell.
\end{align*}
We can estimate for any $\delta>0$
\[
\left|\frac{\partial u}{\partial \nu}\frac{\partial u}{\partial \tau} \right| \le \frac\delta 2 \left|\frac{\partial u}{\partial \nu}\right|^2
+ \frac1{2\delta} \left|\frac{\partial u }{\partial \tau} \right|^2.
\]
Using $|(z-z^*)\cdot\nu|\ge k|z-z^*|\ge k \min_{z\in \partial D} |z-z^*|$ and 
$|(z-z^*)\cdot \tau| \le \sqrt{1-k^2} \max_{z\in \partial D} |z-z^*|$, we arrive at the claim by suitable
choices of $\delta$.
\end{proof}
Now, returning to our domain $\Omega$,
we define a useful quantity related to the radial derivative of the energy in balls around $z_0\in\dOm$, by
\begin{equation}
\calA_{\e,z_0}(\rho):=
\rho \int_{\partial B_\rho(z_0) \cap \Omega} |\nabla u_\e|^2 \,d\ell
+ \frac\rho\e \int_{\partial \Gamma_\rho(z_0)} G(u_\e) d\mathcal{H}^0,
\end{equation}
where $\Gamma_\rho(z_0)=:\dOm\cap B_\rho(z_0)$. For 
small $\rho$, $\partial\Gamma_\rho$ is a set with two elements, hence the integral over this 
set with respect to counting measure $\mathcal{H}^0$ is just a sum of two values.
We will abbreviate $\calA(\rho)=\calA_{\e,z_0}(\rho)$. We recall that $\beta_0$ is chosen such that
\eqref{eq:eb0heto0} holds.

\begin{proposition}\label{prop:mainpohocalc}
There exist $\e_0>0$ and $C>0$ depending only on $\Omega$ 
such that for every $\beta\in (\beta_0,1)$ and
for every $z_0\in \dOm$, 
$\e<\e_0$
and 
$\rho\in(\e,\e^\beta)$, the set
 $D_\rho=\Omega\cap B_\rho(z_0)$ satisfies the assumptions of Lemma
 \ref{lem:comparenutau} with $D=D_\rho$ and a suitable $z_\rho^*\in D_\rho$,
 with constants uniform in $\rho$.
In addition,
\begin{equation} 
\frac1\e \int_{\Gamma_\rho(z_0)} G(u_\e)\,d\ell \le 
2\calA(\rho)+C\e^{\beta-\beta_0}. 
\end{equation}
\end{proposition}
\begin{proof}
 Without loss of generality $z_0=0$.
Using the smoothness of $\dOm$, we can find $Z\in C^1(\overline{D_\rho};\R^2)$ such that $Z\cdot \nu=0$ 
on $\Gamma_\rho$, $Z(0)=0$, $|\nabla Z-\mathrm{id}_{2\times 2}|\le C|z|$, where $C=C(\Omega)$ and $\mathrm{id}_{2\times 2}=\left(\begin{smallmatrix} 1&0\\0&1\end{smallmatrix}\right)$.
We apply \eqref{eq:rellich} to $u=u_\e$ on $D=D_\rho$. It follows that
\begin{align*}
\frac12\int_{\partial D_\rho} 
(z\cdot \nu) |\nabla u_\e|^2 \,d\ell=
\rho \int_{\partial B_\rho \cap \Omega} \left|\frac{\partial u_\e}{\partial \nu}
\right|^2 \,d\ell & + \int_{\Gamma_\rho}  \frac{\partial u_\e}{\partial \nu}
Z\cdot \nabla u_\e \,d\ell \\ & + \int_{\Gamma_\rho} \frac{\partial u_\e}{\partial \nu}
(z-Z)\cdot \nabla u_\e\,d\ell.
\end{align*}
Using the boundary condition in \eqref{eq:bc1} and 
$Z\cdot \nabla u_\e=(Z\cdot \tau) \frac{\partial u_\e}{\partial\tau}$, 
we obtain after integration by parts
\begin{align*}
\int_{\Gamma_\rho} \frac{\partial u_\e}{\partial \nu} Z\cdot \nabla u_\e\,d\ell
&= \frac{1}{\e} \int_{\Gamma_\rho} f(u_\e) \frac{\partial u_\e}{\partial \tau}
(Z\cdot \tau)\,d\ell + \int_{\Gamma_\rho} h_\e \frac{\partial u_\e}{\partial \tau}
(Z\cdot \tau)\,d\ell \\
&=-\frac1\e \int_{\partial\Gamma_\rho} G(u_\e) |Z\cdot \tau| d\mathcal{H}^0
+\frac1\e \int_{\Gamma_\rho} G(u_\e) \frac{\partial}{\partial\tau}(Z\cdot\tau)\,d\ell
\\ & \quad + \int_{\Gamma_\rho} h_\e \frac{\partial u_\e}{\partial \tau}
(Z\cdot \tau)\,d\ell.
\end{align*}
Hence
\begin{align*}
\frac1\e \int_{\Gamma_\rho} G(u_\e) \frac{\partial}{\partial\tau}(Z\cdot\tau)\,d\ell
&=\frac1\e \int_{\partial\Gamma_\rho} G(u_\e) |Z\cdot \tau| d\mathcal{H}^0
-\int_{\Gamma_\rho} h_\e \frac{\partial u_\e}{\partial \tau}
(Z\cdot \tau)\,d\ell\\
&\quad+
\frac12\int_{\partial D_\rho} 
(z\cdot \nu) |\nabla u_\e|^2\,d\ell-\rho \int_{\partial B_\rho \cap \Omega} \left|\frac{\partial u_\e}{\partial \nu}
\right|^2\,d\ell  \\ & \quad  -\int_{\Gamma_\rho} \frac{\partial u_\e}{\partial \nu}
(z-Z)\cdot \nabla u_\e\,d\ell.
\end{align*} 

Using the bound on $Z$ and the fact that $|z\cdot \nu|\le C\rho^2$ on $\Gamma_\rho$, 
it follows that
\begin{align*}
\left(1-C\rho
 \right) \frac1\e \int_{\Gamma_\rho} & G(u_\e)\,d\ell
 \le (1+C\rho)\frac\rho\e \int_{\partial\Gamma_\rho} G(u_\e) d\mathcal{H}^0
 +C\rho^2 \int_{\Gamma_\rho} |\nabla u_\e|^2\,d\ell \\
 & \, \, +\left(\frac{\rho}{2}+C\rho^2+C\rho \sup|h_\e|\right) \int_{\partial B_\rho\cap\Omega}
 |\nabla u_\e|^2\,d\ell + \rho\sup|h_\e|.
\end{align*}
By Lemma \ref{lem:comparenutau}, 
\[
\int_{\Gamma_\rho} |\nabla u_\e|^2\,d\ell \le C\left(
\int_{\Gamma_\rho} \left|\frac{\partial u_\e}{\partial \nu} \right|^2\,d\ell
+ \int_{\partial B_\rho\cap \Omega} |\nabla u_\e|^2\,d\ell
\right).
\]
We estimate 
\[
\int_{\Gamma_\rho}\left|\frac{\partial u_\e}{\partial \nu} \right|^2\,d\ell
\le C\left( \rho \sup|h_\e|^2 + \frac1{\e^2}\int_{\Gamma_\rho} |f(u_\e)|^2\,d\ell\right).
\]
We recall that 
$|f|^2 \le CG$ for some constant, and obtain
\[
\int_{\Gamma_\rho}\left|\frac{\partial u_\e}{\partial \nu} \right|^2\, d\ell
\le C\left(\rho \sup|h_\e|^2 + \frac1\e\left(\frac1\e\int_{\Gamma_\rho} G(u_\e)\, d\ell \right)  \right).
\]
Combining all terms, we arrive at the estimate
\begin{multline*}
\left(1- C\rho -C\frac{\rho^2}{\e}- C\rho^3 \sup |h_\e|^2
 \right)\frac1\e\int_{\Gamma_\rho} G(u_\e)\, d\ell
 \\
 \le \rho \sup|h_\e|+ 
 \left(\frac{\rho}{2}+C\rho^2+C\rho \sup|h_\e|\right) \int_{\partial B_\rho\cap\Omega}
 |\nabla u_\e|^2\, d\ell
 +(1+C\rho)\frac\rho\e \int_{\partial\Gamma_\rho} G(u_\e) d\mathcal{H}^0.
\end{multline*}

For $\e<\rho<\e^\beta$, $1>\beta>\beta_0$, it follows that 
\begin{align}
\frac1\e \int_{\Gamma_\rho} G(u_\e)\, d\ell & \le 2
\rho \int_{\partial B_\rho \cap \Omega} |\nabla u_\e|^2 \, d\ell
+ 2\frac\rho\e \int_{\partial \Gamma_\rho} G(u_\e) d\mathcal{H}^0
+C\e^{\beta-\beta_0} \notag \\ & =2\calA(\rho)+C\e^{\beta-\beta_0}
\end{align}
for $\e<\e_0$ and 
$\e_0$ sufficiently small (depending only on $\Omega$ and $G$), as desired.
\end{proof}

The use of the quantity $\calA(\rho)$ becomes clear with the following lemma, showing that 
if it is small, then $u_\e$ is away from zero in a smaller region.
\begin{lemma}\label{lem:UpsG}
Under the assumptions of Proposition \ref{prop:mainpohocalc},
there exist positive constants $\kappa$ and $C$ independent of $\e$ for which the following holds.
If
\(
\calA(\rho)<\kappa
\) for some $\rho>0$ 
then
\[
\inf_{\Gamma_{\rho/2}(z_0)} |u_\e|>t_*
\]
and 
\[
\frac1\e\int_{\Gamma_\rho(z_0)} G(u_\e)\, d\ell\le C.
\]
\end{lemma}
\begin{proof}
If $|u_\e(z)|\le t_*$ 
 for some  
$z\in \Gamma_{\rho/2}(z_0)$, it
follows by the bound $|\nabla u_\e|\le \frac C\e$
 that
$|u_\e(z)| \le \frac12(1+t_*)$ for $z\in \Gamma_{c\e}(z_0)\subset \Gamma_{\rho}(z_0)$,
and thus there is a $C_G>0$ with
\[
\frac1\e \int_{\Gamma_\rho(z_0)} G(u_\e)\, d\ell\ge c C_G
\]
For $\kappa<c C_G$, this leads to a contradiction. 
The second statement is clear from Proposition \ref{prop:mainpohocalc}.
\end{proof}
In the following, we use the localized energies for $S\subset \overline{\Omega}$ given by 
\[
E_\e(u;S) = \frac12 \int_{S\cap \Omega} |\nabla u|^2 \, dxdy + \frac1\e \int_{{\overline{S}}\cap \dOm} G(u)\, d\ell.
\]
The logarithmic bound on the energy allows us to find radii where $\mathcal{A}(\rho)$ is not too large.
\begin{lemma}\label{lem:dis}
Let $1>\alpha_2>\alpha_1>\alpha$, $L>0$. Then for every $z_0\in \dOm$ and every
$\e>0$ sufficiently small, we have
\begin{equation}\label{eq:essinf}
\inf_{L\e^{\alpha_2}<r<L\e^{\alpha_1}} \calA(r)
\le \frac{2}{(\alpha_2-\alpha_1)|\log\e|} E_\e(u_\e;\ol\Omega\cap B_{L\e^{\alpha_1}}(z_0))
 \le \frac{2K}{\alpha_2-\alpha_1}.
\end{equation}
\end{lemma}
\begin{proof}
We have that 
\[
K|\log\e|\ge \frac12\int_{\overline{\Omega}\cap B_\rho(z_0)}| \nabla u_\e|^2dxdy + 
\frac1\e\int_{\Gamma_\rho(z_0)} G(u_\e) d\ell
\ge \frac12\int_{0}^\rho \frac{\calA(s)}{s}ds
\]
Restricting the integration range and estimating $\calA\ge \inf\calA$,
it follows that
\[
K|\log\e| \ge E_\e(u_\e;\overline{\Omega}\cap B_\rho(z_0)) \ge
\frac{\inf \calA}{2} \log \frac{L\e^{\alpha_1}}{L\e^{\alpha_2}} = \frac{\inf \calA}{2} |\log\e| (\alpha_2-\alpha_1),
\]
which gives the claim.
\end{proof}
We now define a ``bad set'' and show that it can be covered by a bounded number of $\e$-balls.
\begin{proposition}\label{prop:finitecoverprop}
There exists a number $N$ independent of $\e$ such that the 
approximate transition set
\[
S_\e=\{ x\in \dOm: |u_\e| \le t_*\}
\]
can be covered by at most $N$ balls of radius $\e$. 
\end{proposition}
\begin{proof}
We choose $1>\beta_3>\beta_2>\beta_1>\beta_0$.
For every $z\in S_\e$, there exists by Lemma \ref{lem:dis} and Lemma \ref{lem:UpsG}
a radius $r\in [\e^{\beta_3},\e^{\beta_2}]$ such that
\[
\kappa\le \calA(r)\le\frac{2E_\e(u_\e;\overline{\Omega}\cap B_r(z)) }
{(\beta_3-\beta_2)|\log\e|}. 
\]
Using Vitali's covering lemma, we can choose a set 
$\{\zeta_\e^j : j\in J_\e\}\subset S_\e$ such that $B_{\e^{\beta_2}}(\zeta_\e^j)$ 
are disjoint and $S_\e\subset B_{5\e^{\beta_2}}(\zeta_\e^j)$. It follows that
\[
|J_\e|\kappa \le \frac{2K}{\beta_3-\beta_2}.
\] 
Now we choose radii $\rho_j\in [5\e^{\beta_2},5\e^{\beta_1}]$ with the
property that 
\[
\calA(\rho_j) = \calA_{u_\e, \e, \zeta^j_\e}(\rho_j) \le \frac{2K}{\beta_2-\beta_1} .
\]
By Proposition \ref{prop:mainpohocalc}, this leads to
\begin{equation}\label{eq:need53}
\frac1\e \int_{\Gamma_{\rho_j}(\zeta^j_\e)} G(u_\e)\, d\ell \le C.
\end{equation}

Using once more Vitali's covering lemma, we can choose $z_\e^k\in S_\e$, 
$k=1,\dots, N_\e$ such that $B_{\e/5}(z_\e^k)$ are disjoint and 
$S_\e\subset \bigcup_{k=1}^{N_\e} B_\e(z_k)$. Since $z_\e^k\in S_\e$,
it follows that \[\frac1\e\int_{\Gamma_{\e/5}(z_\e^k)} G(u) \ge c>0\] and so
\[
c N_\e \le \frac1\e\int_{\bigcup_{k=1}^{N_\e} \Gamma_{\e/5}(z_\e^k)}
G(u_\e) \, d\ell \le C|J_\e|
\le \frac{2CK}{\kappa(\beta_3-\beta_2)}.
\]
This gives the claim for $N=\limsup_{\e\to 0} N_\e$. 
\end{proof}

The balls we have found so far are not disjoint. Similar to arguments in \cite{BethuelBrezisHelein:1994a}, 
we will show that they can be replaced by larger, disjoint balls. We find it more convenient to work in the 
half-plane. Thus,
we now use conformal mapping to transfer the problem to a half-plane. In the half-plane, we cover the transition set by larger balls of size $M\e$ (for a constant $M$ that depends on the sequence $(u_\e)$, but is not of great importance; we will 
find sharper bounds later) that are 
disjoint and at mutual distances much larger than $\e$. In analogy with the Ginzburg-Landau situation, we continue
to use the word ``balls'' although they are one-dimensional intervals.
\begin{proposition}\label{pro:57}
For any sequence $\e\to 0$ there is a subsequence, a point $p_\infty$ on the oriented segment from $q$ to $p$ on $\dOm$,
and a constant $M>0$ independent of $\e$ such that the following holds. If $\varphi$ is a conformal map as in Remark~\ref{rem:crossr} with inverse $\psi$
and setting $\tilde u_\e = u_\e\circ \psi$,
the set
\[
\tilde S_\e = \left\{ x\in \R : \left|\tilde u_\e(x)\right|\le t_* \right\}
\]
has a cover by $N$ disjoint balls of radius $M\e$ whose centers $x^k_\e$ converge to points $x^k_0\in\R$ as $\e\to 0$ and whose distances satisfy $\frac{|x^k_\e-x^j_\e|}{\e} \to \infty$ 
for $j\neq k$. Here $N$ is the number from the previous proposition.

Additionally, there is $\beta\in(0,1)$ independent of $\e$ such that 
\begin{equation}\label{eq:penbd57}
\frac1\e\int_{(x^k_\e-\e^\beta,x^k_\e+\e^\beta)} G(u_\e) d\ell \le C
\end{equation}
\end{proposition}
\begin{proof}
The compactness of $\dOm$ gives the existence of a subsequence (not labeled) such 
 that $S_\e\subset \bigcup_{k=1}^N B_\e(z^k_\e)$ with $z^k_\e\to z^k_0\in\dOm$.
We can now choose a conformal map $\varphi:\Omega\to\R^2_+$ such that $\varphi(p)=0$, $\varphi(q)=1$ and such that the 
point $p_\infty$ on the oriented segment from $q$ to $p$ with $\varphi(p_\infty)=\infty$ satisfies $p_\infty\neq z^k_0$, $k=1,\dots, N$.
We set $x^k_{\e}=\varphi(z^k_\e)$. By relabeling, we can assume
\[
x^1_\e<x^2_\e<\dots<x^N_\e.
\]
We now choose a subsequence such that
\[
\frac{x^2_\e-x^1_\e}{\e} \to \liminf_{\e\to 0}\frac{x^2_\e-x^1_\e}{\e}=C_1 \in(0,+\infty]
\]
Choosing subsequences finitely many times, we may assume
\[
\frac{x^{k+1}_\e-x^k_\e}{\e} \to \liminf_{\e\to 0}\frac{x^{k+1}_\e-x^k_\e}{\e}=:C_k \in(0,+\infty]
\]
for every $k=1,\dots, N-1$.

If $C_k<\infty$, this means that in our chosen subsequence,  $x^{k+1}_\e$ and $x^k_\e$ are at a mutual distance of at most
$(1+C_k)\e$. They are the centers of images of $\e$-balls under $\varphi$.  For $C_k=\infty$, 
$x^{k+1}_\e$ and $x^k_\e$ are at a mutual distance that is larger than $M\e$, for any $M$.

We set
\[M=1+\sum_{k=1}^{N} |\varphi'(z^k_0)|+\sum_{C_k\neq \infty} C_k.\] 
We can now choose from $\{1,\dots, N\}$ a subset $\mathcal{K}$ such that balls with these centers satisfy the claim
by  taking just one representative for all subsets of points that have mutual distances smaller than 
$M\e$. We set 
\[
\mathcal{K} = \{N\} \cup \bigl\{ k \in\{1,\dots, N-1\} : C_{k}=\infty\bigr\}.
\]
We now have for $\e$ small enough in our subsequence
\[
\tilde S_\e \subset \bigcup_{k\in \mathcal{K}} B_{M\e}(x^k_\e),
\]
which is a disjoint union of balls at distance $\gg \e$. We clearly have $x^k_\e\to x^k_0 = \varphi(z^k_0)$. 

The bound \eqref{eq:penbd57} follows from \eqref{eq:need53}.
\end{proof}

\section{Blow-up and the sharp lower bound}\label{sec:blowup}
In this section, we exploit the covering results of the previous one to allow us to locally blow-up the solution 
near the covering sets. As the 
blow-up limits will be solutions of the half-plane problem, we will use our classification result Theorem~\ref{thm:uni} to 
further reduce the number of covering sets, by ruling out ``homoclinic'' balls where $u_\e$ does not exhibit 
a sign change. We will then return to studying the minimization problem in the closed neighborhood
 $\calC_a^\e$
 and show that there must be exactly two ``heteroclinic'' balls with a sign change. We then proceed to find sharp lower bounds for the energy.

Our assumptions continue to be the same as in the previous section, but we now work mostly in $\R^2_+$ thanks to the 
conformal map $\varphi:\Omega\to \R^2_+$ with inverse~$\psi$. The harmonic functions $\tilde u_\e=u_\e\circ \psi$ are defined on $\R^2_+$,
and we now consider their blow-up limits.

\begin{proposition}\label{pro:61}
Let the subsequence, the set $\mathcal{K}$, the points $x_\e^k$ and $M>0$ be chosen as in Proposition~\ref{pro:57}. Choose
$k\in \mathcal{K}$ and
consider the functions $\tilde U_\e$ given by
\[
\tilde U_\e(x,y) = 
\tilde u_\e(x_\e^k+\e x/s_k, \e y/s_k) ,
\]
where $s_k=|\psi'(x_0^k)|$.

Then, for a subsequence, $\tilde U_\e\to U$ locally uniformly in the closure of $\R^2_+$ and in $H^1(B_R^+)$ for every $R>0$,
where $U$ is either one of the constants $\pm 1$ 
or a layer solution.
\end{proposition}
\begin{proof}
The function $\tilde U_\e$ satisfies
\[
\begin{cases}
\Delta \tilde U_\e&=0 \quad\text{in $\R^2_+$}\\
\frac{\partial \tilde U_\e}{\partial \nu} &= b_\e(x) f(\tilde U_\e)+ r_\e(x)\quad \text{on $\R$},
\end{cases}
\]
where $b_\e(x)=\frac1{s_k}\psi'(x^k_\e+\frac{\e}{s_k}x)\to 1$ locally uniformly on $\R$ as $\e\to 0$,  and also
\[r_\e(x)=\frac{\e}{s_k} (h_\e \circ \psi)(x^k_\e+\frac{\e}{s_k}x)
\to 0\] locally uniformly by our assumption \eqref{eq:eb0heto0}.

By \eqref{eq:gradb}, we have $|\nabla \tilde U_\e|\le C$, and hence by Arzel\`a--Ascoli we have (for a subsequence)  that $\tilde U_\e\to U$ locally uniformly on $\ol{\R^2_+}$.
Also, $\tilde U_\e\to U$ locally weakly in $H^1$.

This gives that we can pass to the limit in the weak formulation of the PDE and obtain that $U$ solves
\begin{align*}
\begin{cases}
\Delta U&=0 \quad\text{in $\R^2_+$}\\
\frac{\partial U}{\partial \nu} &= f(U)\quad \text{on $\R$}.
\end{cases}
\end{align*}
Testing the difference of the PDEs with $\eta^2(\tilde U_\e-U)$ for suitable cutoff functions $\eta$,
we find that 
\[
\int_{\R^2_+} \nabla (\eta^2(\tilde U_\e -U)) \cdot \nabla (\tilde U_\e-U) \,dxdy
+ \int_\R \eta^2(\tilde U_\e -U) (b_\e f(\tilde U_\e)+ r_\e -f(U)) dx.
\]
Hence after integrating by parts and using standard estimates,
\begin{align*}
\int_{\R^2_+} \eta^2 |\nabla( \tilde U_\e-U)|^2 \, dxdy
& \le C \int_{\R^2_+} |\nabla \eta|^2 | \tilde U_\e-U|^2 \, dxdy \\ & \qquad + C  \int_\R \eta^2(\tilde U_\e -U) (b_\e f(\tilde U_\e)+ r_\e -f(U)) dx.
\end{align*}
From the locally uniform convergence, 
we see that for $\eta\in C^1$ with compact support in the closed upper half-plane, we must have 
\[
\int_{\R^2_+} \eta^2 |\nabla( \tilde U_\e-U)|^2 \, dxdy \to 0
\]
and hence we obtain strong $H^1$ convergence locally. 

By the uniform convergence we see that $|U(x,0)|\ge t_*$ for $|x|>CM$, for a constant $C$ related to the rescaling $s_k$.
From \eqref{eq:penbd57} and Lebesgue's convergence theorem we note that $\int_\R G(U)\, dx<\infty$. 
As $|\nabla U|\le C$, we must have that $U\to 1$ or $U\to -1$ as $x\to\pm\infty$. 
 By Theorem~\ref{thm:uni}, it follows that $U$ is a layer solution if the limits at $\pm \infty$ are different (the ``heteroclinic'' 
 case), or $U$ is one of the constants $\pm 1$ if the limits are the same (``homoclinic''). 
\end{proof}

The uniform convergence  in the previous proposition allows us to obtain further information about our
cover of the  ``bad set'' $\tilde S_\e$  of approximate transitions: for sufficiently small $\e$, the balls of 
size $M\e$ in the cover must  correspond to sign changes of $\tilde u_\e$, as any hypothetical 
balls without sign changes are disjoint from $\tilde S_\e$. We collect our results as follows.
\begin{corollary}\label{cor:62}
For a subsequence, there is a cover of $\tilde S_\e$ with a bounded number of  balls of radius $M\e$
that are disjoint, have a mutual distance that is large compared to $\e$, and 
such that $\tilde u_\e$ changes sign at the end points of the intervals. The upper bound for the number of these balls depends only on the energy bound, the nonlinearity, 
and the domain.
\end{corollary}
We can improve this significantly for minimizers in $\calC_a^\e$,
 where we will show that exactly two $M\e$  balls suffice to cover the set $\tilde S_\e$.
 First, we prove a lemma to compute a lower bound in an annulus where $\tilde u_\e$ has a sign change. Recall that the notation $A_{S,R}$, $A_{S,R}^0$ was defined in 
\eqref{eq:notationA}. 

\begin{lemma}\label{lem:l2} Let $R_0>0$ be given.
Let $u\in H^1_{\mathrm{loc}}(\R^2_+)$ with $|u|\le 1$. If $x_0\in\R$ and $[S,R]\subset(0,R_0]$ is an interval such that
$|u(x_0+r,0)|\ge t_*$ and $u(x_0+r,0)u(x_0-r,0)<0$ for $r\in[-R,-S]\cup[S,R]$, then
\[
\frac12\int_{x_0+A_{S,R}^+}  |\nabla u|^2 \, dxdy + \frac1\e \int_{x_0+A^0_{S,R}} |\psi'(x)| G(u)\, dx
\ge \frac2\pi \log \frac RS -C(R_0) \frac \e S,
\]
where $C(R_0)$ is a constant depending only on $R_0$ and $\psi$. 
\end{lemma}
\begin{proof}
As the trace of $u$ does not have any jump singularities, it follows that the sign of $r\mapsto u(x_0+r)$ is constant in $[S,R]$ and 
in $[-R,-S]$. Without loss of generality we may assume that $x_0=0$ and that $u$ has the same sign as $\frac2\pi \theta^0-1$.
From $|u|\ge t_*$, we find that 
\[
G(u) \ge c\left|u-\frac2\pi \theta^0+1\right|^2.
\]
Setting $m=\min_{r\in[-R,R]} |\psi'(r)|$ and $v=|u-\frac2\pi \theta^0+1|$, we can now estimate using Lemma~\ref{l:lem1}
that 
\[
\int_{A^+_{S,R}} \frac12|\nabla u|^2 \, dx dy+ \frac1\e \int_{A^0_{S,R}} |\psi'(x)| G(u)\, dx  \ge
\frac2\pi \log \frac R S - \int_{A^0_{S,R}} \left(- \frac{2v}{\pi r} + \frac {mc v^2}{\e} \right)\, dr
\]
and optimizing over $v$, we find that $- \frac{2v}{\pi r} + \frac {mc v^2}{\e }  \ge -\frac{\e}{mc\pi^2 r^2}$.
After integrating over $r$, we find that
\[
\int_{A^+_{S,R}} \frac12|\nabla u|^2 \, dx + \frac1\e \int_{A^0_{S,R}} |\psi'(x)| G(u)\, d\ell  \ge\frac2\pi \log \frac RS
 -C\frac\e m \left(\frac1S -\frac1R \right) ,
\]
which gives the claim. 
\end{proof}
Now we can show the following strengthening of Corollary~\ref{cor:62}. Our proof is inspired by 
an argument of Struwe in the proof of \cite[Proposition 3.3]{Struwe:1994a}.
\begin{proposition}\label{pro:just2}
If $u_\e$ are  minimizers in the closed neighborhood $\calC_a^\e$ as in Proposition~\ref{pro:conmin}, then for a subsequence, the 
 set $\tilde S_\e$ can be covered with exactly two balls of radius $M\e$.
\end{proposition}
\begin{proof}
With the set $\mathcal{K}$, the points $x_k^\e$ and $M>0$ be chosen as in Proposition~\ref{pro:57}, we can assume, possibly taking a subset 
$\mathcal{K}_1\subset \mathcal{K}$,  by Proposition~\ref{pro:61} and
Corollary~\ref{cor:62} that $\tilde u_\e(x_k^\e-M\e,0) \tilde u_\e(x_k^\e+M\e,0) <0$ while
\[
\tilde S_\e\subset \bigcup_{k\in \mathcal{K}_1} [x_k^\e-M\e,x_k^\e+M\e].
\]
Let $\alpha_1=\alpha_1(\e)=M\e$ and $\beta_1=\beta_1(\e)=\frac12 \min_{j<k\in \mathcal{K}_1} (x_k^\e-x_j^\e)\gg \e$. 
Then for
$r\in(\alpha_1,\beta_1)$, 
the sets  $F(k,\e,r)=[x_k^\e-r,x_k^\e+r]$, $k\in\mathcal{K}_1$ are pairwise disjoint and satisfy $\{x_k^\e-r,x_k^\e+r\}\cap \tilde S_\e=\emptyset$
while $\bigcup_{k\in \mathcal{K}_1} [x_k^\e-r,x_k^\e+r]\supset \tilde S_\e$.
For $r=\beta_1$, the sets $F(k,\e,r)$ are no longer pairwise disjoint.

By alternatingly doubling $r$ and removing indices $k$ where $k'\neq k$ exist with $F(k,\e,r)\cap F(k',\e,r)\neq \emptyset$ (which corresponds to merging the
corresponding intervals into a larger one), we find 
$\alpha_2\le C\beta_1$ and $\mathcal{K}_2\subsetneq \mathcal{K}_1$
with the property that the intervals 
$[x_k^\e-\alpha_2,x_k^\e+\alpha_2]$, $k\in\mathcal{K}_2$ are mutually disjoint and satisfy 
\begin{equation}\label{eq:intergr}
S_\e\bigcup_{k\in \mathcal{K}_1} [x_k^\e-\alpha_1,x_k^\e+\alpha_1]
\subset
\bigcup_{k\in \mathcal{K}_2} [x_k^\e-\alpha_2,x_k^\e+\alpha_2].
\end{equation}
We now set $\beta_2=\frac12 \min_{j<k\in \mathcal{K}_2} (x_k^\e-x_j^\e)$.
Then for 
$r\in(\alpha_2,\beta_2)$, the intervals $[x_k^\e-r,x_k^\e+r]$, $k\in\mathcal{K}_2$ are mutually disjoint,
their union covers the union the intervals from the first step and from \eqref{eq:intergr} we see that
there holds $\{x_k^\e-r,x_k^\e+r\}\cap \tilde S_\e=\emptyset$. Hence, we have 
$|\tilde u_\e|\ge t_*$ in each of the sets $(x_k^\e-\beta_2,x_k^\e-\alpha_2)\cup(x_k^\e+\alpha_2,x_k^\e+\beta_2)$.
 While $\alpha_2$ and $\beta_2$ depend on $\e$, we can take a 
 subsequence such that $\mathcal{K}_2$ is independent of $\e$.

We can repeat this procedure and find intervals 
$(\alpha_\ell,\beta_\ell)$, $1\le \ell\le m$ (which depend on $\e$) and 
subsets $\mathcal{K}_m\subsetneq \dots \subsetneq 
\mathcal{K}_1$ (which do not depend on $\e$)
such that for $r\in (\alpha_\ell,\beta_\ell)$, the intervals $[x_k^\e-r,x_k^\e+r]$, $k\in\mathcal{K}_\ell$ are mutually disjoint, cover the intervals of the previous steps, and their end points are not in $\tilde S_\e$, 
 i. e.  $\{x_k^\e-r,x_k^\e+r\}\cap \tilde S_\e=\emptyset$. We will stop the growth and merging procedure when 
 $\beta_m\ge \frac14$.

Note that by construction, 
\[
\mathcal{A}=\bigcup_{\ell=1}^m \bigcup_{k\in \mathcal{K}_\ell} (x_k^\e+A^+_{\alpha_\ell,\beta_\ell})
\]
is a disjoint union of half-annuli that does not intersect $\tilde S_\e$.

Set 
\[D_\ell=\left | \left\{ k\in \mathcal{K}_\ell :  \tilde u_\e(x_k^\e-\alpha_\ell) \tilde u_\e(x_k^\e+\alpha_\ell) <0 \right\}  \right |\] 
Then $D_\ell$ is the number of intervals at stage $\ell$ such that $\tilde u_\e$ changes sign at their end. By construction, we
have that $D_1\ge  D_2\ge  \dots \ge D_m$. From the fact that $u_\e\in \calC_a^\e$,
we must have two sign changes and so
 we obtain $D_\ell\ge 2$ for all $\ell$. As in the first step, \emph{all} intervals correspond to sign changes and the total number of intervals is 
decreasing, we must have $D_1>D_2$ unless $m=1$. 

Using Lemma~\ref{lem:l2}, we find by integrating over $\mathcal{A}$ and using the upper bound from Proposition~\ref{prop:ub} that
\[
\frac4\pi \log\frac1\e+C
\ge \sum_{\ell} \frac2\pi D_\ell \left( \log \frac{\beta_\ell}{\alpha_\ell} -C\right).
\]
If $D_{\ell_0}>2$ while $D_{\ell_0+1}=2$, we have, using $\alpha_{\ell+1}\le C\beta_\ell$ and telescoping the sum,
\begin{align*}
\frac4\pi \log\frac1\e+C &\ge \frac4\pi \sum_{\ell=1}^m (\log \beta_\ell-\log \alpha_\ell)  -C +(D_{\ell_0}-2) \frac2\pi \sum_{\ell=1}^{\ell_0} (\log \beta_\ell-\log \alpha_\ell) \\
&\ge  \frac 4\pi \log \frac1\e -C + (D_{\ell_0}-2)\frac2\pi \log \frac{\beta_{\ell_0}}{\e}.
\end{align*}
Thus we must have $D_{\ell_0}=2$ or $\beta_{\ell_0}\le C\e$. However, the latter is impossible since $\beta_{\ell_0} \ge \beta_1\gg \e$,
and it follows that $m=1$ and $D_1=2$. 
\end{proof}
An important corollary of Proposition~\ref{pro:just2} is the following.
\begin{corollary}\label{cor:pq}
For a subsequence, we have  $u_\e\to \chi^{p_0,q_0}$ in $L^2(\dOm)$, where $p_0$ and $q_0$ 
can be found as follows. Considering the centers of the intervals covering $\tilde S_\e$ from the previous
computation, we can assume 
 $x_1^\e\to x_1$ and $x_2^\e\to x_2$ and set
 $p_0=\psi(x_1)$, $q_0=\psi(x_2)$. By construction (recall Lemma~\ref{lem42}) and the triangle inequality, we see (after possibly swapping $p_0$ and $q_0$ so that the sign is consistent)
\[
\|\chi^{p_0,q_0}-\chi^{p,q}\|_{L^2(\dOm)} \le \|\chi^{p_0,q_0}-u_\e\|_{L^2(\dOm)} + \|u_\e-\chi^{p,q}_\e\|_{L^2(\dOm)}
+ \|\chi^{p,q}-\chi^{p,q}_\e\|_{L^2(\dOm)}.
\]
Hence passing to the limit $\e\to 0$,
\begin{equation}\label{eq:p0clp}
|p_0-p|+|q_0-q| \le Ca,
\end{equation}
where we recall $a$ from the definition of our constraint set $\calC_a^\e$.

The strong $L^2(\partial\Omega)$ convergence follows immediately by noticing that we have  $|u_\e-\chi^{p_0,q_0}|\le 2$ in two sets of length $O(\e)$
and $|u_\e-\chi^{p_0,q_0}|^2\le cG(u)$ elsewhere, and hence
\[
\int_{\partial\Omega} |u_\e-\chi^{p_0,q_0}|^2 d\ell \le C\e + c\int_{\partial\Omega} G(u) d\ell \le C\e + cK\e \log\frac1\e \to 0
\]
as $\e\to 0$.
\end{corollary}

In the following, we use a slightly different conformal map $\hat \psi$ instead of $\psi$, namely we choose $\hat\psi$ such that $\hat \psi(0)=p_0$ and $\hat \psi(1)=q_0$,
while $\hat \psi(\infty)=p_\infty=\psi(\infty)$. 
(Recall that $\psi(0)=p$ and $\psi(1)=q$).
We thus can assume that $x_1$ and $x_2$, the limits of the centers of the balls covering the approximate transition set for $\tilde u_\e$,
are just the points
 $x_1=0$ and $x_2=1$. 
If $x_1^\e$ and $x_2^\e$ are zeroes of $\tilde u_\e$, then they lie in $\tilde S_\e$, and
we can cover this set  with balls of size $2M\e$ around the zeroes.

With this setup, we now compute asymptotic lower bounds for the energy of a minimizer in
 $\calC_a^\e$ over
 various subdomains.
\begin{proposition}\label{prop:inner}
Let $x_k^\e\in\R$, $k=1,2$ be points such that  $\tilde u_\e(x_k^\e)=0$ and $\tilde S_\e\subset \bigcup_{k=1,2} [x_k^\e-2M\e,x_k^\e+2M\e]$.
For any $R\ge 2M$, the energy of $\tilde u_\e$ satisfies
\[
\int_{B_{R\e}^+(x_k^\e)} \frac12|\nabla \tilde u_\e|^2\, dx dy+ \int_{[x_k^\e-R,x_k^\e+R]} \frac1\e |\hat \psi'(x)| G(\tilde u_\e(x))\, dx
\to \frac2\pi \log (Rs_k) + C_f,
\]
as $\e\to 0$
where $s_k = \lim_{\e\to 0} |\hat \psi'(x_k^\e)|=|\hat \psi'(\lim_{\e\to 0} x_k^\e)|$.  
\end{proposition}
\begin{proof}
This follows from the strong convergence in Proposition~\ref{pro:61} (note that the limit must be the layer solution) 
and the convergence of the energy of the
layer solution in Proposition~\ref{prop:i_er}.
\end{proof}
\begin{proposition}\label{prop:middle}
For $R\ge 2M$ and $R\e<\rho<\frac14$ we have
\begin{multline*}
\int_{B_\rho^+(x_k^\e)\setminus B^+_{R\e}(x_k^\e)}\frac12|\nabla \tilde u_\e|^2\, dxdy + \int_{(x_k^\e-\rho,x_k^\e+\rho)\setminus[x_k^\e-R\e,x_k^\e+R\e]} \frac1\e |\hat \psi'(x)| G(\tilde u_\e(x))\, d\ell 
\\ \ge \frac2\pi \log \frac{\rho}{R\e} - \frac C R
\end{multline*}
\end{proposition}
\begin{proof}
This is a direct consequence of Lemma \ref{lem:l2}.
\end{proof}
\begin{proposition}\label{prop:w1p}
For $0<\rho<\frac14$ and $\tilde \Omega_{\rho}=\R^2_+\setminus (\overline{B_\rho(0)}\cup \overline{B_\rho(1)})$, 
the functions $\tilde u_\e$ satisfy for $\e<\e_0(\rho)$ 
\[
\frac12 \int_{\tilde \Omega_\rho} |\nabla \tilde u_\e|^2 \, dxdy \le \frac4\pi \log \frac 1\rho + C.
\]

In particular, $\tilde u_\e$ converge weakly in $H^1_{\mathrm{loc}}$ to $\tilde u_0$.
\end{proposition}
\begin{proof}
For $\e$ sufficiently small, $B_{\rho/2}(x_k^\e)\subset B_\rho(x_k)$. 
The lower bounds of 
Propositions~\ref{prop:inner} and~\ref{prop:middle}
give for $k=1,2$ that
\[
\int_{B_{\rho/2}^+(x_k^\e)}\frac12|\nabla \tilde u_\e|^2\, dx dy+ \int_{(x_k^\e-\rho/2,x_k^\e+\rho/2)} \frac1\e |\hat \psi'(x)| G(\tilde u_\e(x))\, d\ell 
 \ge \frac2\pi \log \frac{\rho}{2\e} - C.
\]
Together with the upper bound on the energy, we 
find 
\[
\frac12 \int_{\tilde \Omega_\rho} |\nabla \tilde u_\e|^2 \, dxdy \le \frac4\pi \log \frac 1\e + C-
\frac4 \pi \log \frac\rho{2\e}+C,
\]
which leads to the claimed upper bound. 
\end{proof}
\begin{proposition}\label{pro:finallb}
We have
\[
E_\e(u_\e) \ge \frac4\pi \log\frac1\e + W_\Omega(p_0,q_0)+ 2C_f -o_\e(1)
\]
as $\e\to 0$.
\end{proposition}
\begin{proof}
By conformal mapping, the energy can be calculated in the half-plane. Using Proposition \ref{prop:w1p} and lower semicontinuity of 
the Dirichlet energy, we obtain from \eqref{eq:dirintO} a contribution in $\tilde \Omega_\rho$ of 
\[
\frac4\pi \log \frac1\rho-o_\e(1).
\]
For $\e$ sufficiently small,  $B_{\rho(1-\rho)}(x_k^\e)\subset B_\rho(x_k)$.
Using Proposition~\ref{prop:middle} and Proposition~\ref{prop:inner} we find
a contribution on $B_{\rho(1-\rho)}(x_1^\e)\cup B_{\rho(1-\rho)}(x_2^\e)$ of
\[
2(\frac2\pi \log \frac{\rho(1-\rho)}{R\e}+C_f)+ \frac2\pi\log ( Rs_1) + \frac2\pi\log (Rs_2).
\]
hence putting everything together we see for every $\rho$ that
\begin{align*}
E_\e(u_\e) \ge \frac2\pi\left( \log\frac{1}{\rho}  +\log\frac{\rho}{R\e} + \log R + C_f \right) & + \frac2\pi \log (|\hat \psi'(0)| |\hat \psi'(1)|)\\ & +\log(1-\rho)-o_\e(1).
\end{align*}
Letting $\rho\to 0$ we obtain the result.
\end{proof}

\section{Domains where the  renormalized energy has nontrivial local minimizers}\label{sec:localmin}

The renormalized energy $W_\Omega(p,q)$ tends to $-\infty$ when $p\to q$. In this section we show that under certain conditions it has 
further (nontrivial) local minimizers. In fact, we construct smooth domains such that there are arbitrarily many isolated  local minimizers.

The intuition behind our construction is that computing the energy of a harmonic function with values in $\Ss^0$ and a jump in a convex corner gives a larger
contribution to the leading order of the energy than a jump in the middle of a smooth section of a domain with corners. By considering domains that are 
close to polygons, we can find smooth domains where the renormalized energy becomes large if one of the points is near a corner, and small elsewhere. We start by considering a rectangular domain $(0,L)\times(0,H)$
where we can use an explicit series expression to compute the renormalized energy. 
\begin{theorem}\label{thm:rectangle}
In the rectangle $\Omega=(0,L)\times(0,H)$ with $L\le H$, the renormalized energy $W_\Omega(p,q)$
has an isolated local minima at $(p,q)=((L/2,0),(L/2,H))$.
\end{theorem}
\begin{proof}
We use the characterisation of the renomalized energy in terms of the Green's function 
from Proposition~\ref{propo:altren}. It is well known that the Dirichlet Green's function for $\Omega$ 
can be written for $(p,q)=((x,y),(\tilde x,\tilde y))$ as 
\[
G(x,y;\tilde x, \tilde y) = -\frac2\pi \sum_{n=1}^\infty  
\frac{\sin\frac{n\pi x}{L}\sin \frac{n\pi \tilde x}{L}}{n \sinh \frac{n\pi H}{L}} \cdot
\begin{cases}
\sinh \frac{n\pi(\tilde y-H)}{L} \sinh\frac{n\pi y}{L} & y<\tilde y \\
\sinh \frac{n\pi( y-H)}{L} \sinh\frac{n\pi \tilde y}{L} & y >\tilde y. \\
\end{cases}
\]
We compute $\frac{\partial^2G}{\partial \nu_p \partial \nu_q}$ for $p=(x,0)$ and $q=(\tilde x,H)$. We have that
$\frac{\partial^2G}{\partial \nu_p \partial \nu_q} = -\frac{\partial^2G}{\partial y \partial \tilde y}(x,0,\tilde x,H)=:\phi(x,\tilde x)$.
We find 
\begin{equation}\label{eq:phiforrect}
\phi(x,\tilde x)= \frac{2\pi}{L^2}\sum_{n=1}^\infty \frac{n}{\sinh \frac{n\pi H}{L}}
\sin\frac{n\pi x}{L}\sin \frac{n\pi \tilde x}{L}.
\end{equation}
We compute the derivatives
\[
\phi_x(x,\tilde x)=\frac{2\pi^2}{L^3}\sum_{n=1}^\infty \frac{n^2}{\sinh \frac{n\pi H}{L}}
\cos\frac{n\pi x}{L}\sin \frac{n\pi \tilde x}{L}
\]
(the $\tilde x$ derivative is analogous by symmetry),
\[
\phi_{xx}(x,\tilde x)=\phi_{\tilde x\tilde x}(x,\tilde x)=-\frac{2\pi^3}{L^4}\sum_{n=1}^\infty \frac{n^3}{\sinh \frac{n\pi H}{L}}
\sin\frac{n\pi x}{L}\sin \frac{n\pi \tilde x}{L},
\]
and 
\[\phi_{x\tilde x}(x,\tilde x) = \frac{2\pi^3}{L^4}\sum_{n=1}^\infty \frac{n^3}{\sinh \frac{n\pi H}{L}}
\cos\frac{n\pi x}{L}\cos \frac{n\pi \tilde x}{L}.\]
At $x=\tilde x = \frac L2$ we find $\phi_x(\frac L2,\frac L2)=\phi_{\tilde x}(\frac L2,\frac L2)=0$ 
while 
\[
\phi_{xx}(\frac L2,\frac L2)= \phi_{\tilde x\tilde x}(\frac L2,\frac L 2) 
= -\frac{2\pi^3}{L^4}\sum_{n=1}^\infty \frac{n^3}{\sinh \frac{n\pi H}{L}}
\sin^2\frac{n\pi}{2} <0
\]
and 
\[
\phi_{x\tilde x}(\frac L2,\frac L2) = \frac{2\pi^3}{L^4}\sum_{n=1}^\infty \frac{n^3}{\sinh \frac{n\pi H}{L}}
\cos^2\frac{n\pi}{2}.
\]
The point $(\frac L2,\frac L2)$ then corresponds to an isolated local maximum of $\phi$ if the matrix
$D^2\phi(\frac L2, \frac L2)$ is negative definite, which is true if and only if 
\[
\phi_{x\tilde x}(\frac L2,\frac L2) < -\phi_{xx}(\frac L2,\frac L2).
\]
Using that $\sin^2\frac{n\pi}2 =1$ exactly for odd $n$ and $\cos^2\frac{n\pi}2 =1$
exactly for even $n$, we see
\[
 -\phi_{xx}(\frac L2,\frac L2)-\phi_{x\tilde x}(\frac L2,\frac L2) 
 =
 \frac{2\pi^3}{L^4}\sum_{k=1}^{\infty} \left(\frac{(2k-1)^3}{\sinh\frac{\pi H (2k-1)}{L} }
 -  \frac{(2k)^3}{\sinh\frac{\pi H (2k)}{L} }\right)
\]
This quantity is certainly positive if all the terms in the sum are positive.
As the function $t\mapsto \frac{t^3}{\sinh \alpha t}$ is decreasing if
$\alpha t\ge t_0$, where $t_0= 2.9847\dots$ is the positive solution of $t=3\tanh t$, 
our quantity will be positive if $\frac{\pi H}{L}\ge t_0$, in particular for $H\ge L$. 
Thus in these cases we find that $(\frac L2,\frac L2)$ is an isolated local maximum of $\phi$.

An isolated local maximum of $\phi$ immediately furnishes an isolated local minimum of $W_\Omega$ 
thanks to Proposition~\ref{propo:altren}.
\end{proof}

We now construct domains where the renormalized energy
has many local minimizers. We use that in a polygon, the renormalized energy will tend to plus infinity 
whenever the two points are on non-adjacent edges and one of them tends to a corner. The same still holds
true for some nearby smooth domains.

\begin{theorem}\label{thm:lomi}
For any positive integer $k$, there is a smooth convex domain $\Omega$ such that the renormalized energy $W_\Omega$ for $\Omega$ has at least $k$
isolated  local minimizers
on $\partial\Omega\times\partial\Omega$.
\end{theorem}
\begin{proof}
The idea of the proof is to use a smooth approximation of a convex $N$-gon in the complex plane
and to show that the renormalized energy becomes large close to the vertices. To do so, we use
a fairly explicit representation of the derivative of the conformal map. 

We recall (see e.g. \cite{Nehari:1975a}) that the Schwarz-Christoffel formula for mapping of a half-plane conformally onto a convex polygon with exterior angles 
$\pi\alpha_j$ is given by 
\[
\psi(z) = \int_0^z \frac{dw}{(w-a_1)^{\alpha_1}\dots(w-a_N)^{\alpha_N}}.
\]
By a theorem of E. Study \cite{EStudy}, it follows that for any $b>0$, the image $\Omega = \psi(\{ x+iy : y>b\})$ of the shifted half-plane
$\{ x+iy : y>b\}$ is also convex. 
This~$\Omega$ clearly is a smooth domain. 
We compute the renormalized energy related to~$\Omega$. We will choose $b$ later
to ensure the existence of local minimizers.

Recall that for $p,q\in \dOm$ we have
$$ W_{\Omega}(p,q)=\frac{2}{\pi}\, \log \frac{|\varphi(p)-\varphi(q)|^2}{|\varphi'(p)||\varphi'(q)|}
\, .$$
Transforming this to $P,Q\in\R$ such that $\psi(P+ib)=p$, $\psi(Q+ib)=q$, 
we compute
\[
W_\Omega(p,q)=\tilde W(P,Q)= \frac2\pi \log \left(|P-Q|^2 |\psi'(P+ib)||\psi'(Q+ib)|\right).
\]
It suffices to show that $(P,Q)\mapsto \tilde W(P,Q)$ has isolated  local minimizers in $\R\times \R$. 
To simplify matters, we choose as the pre-vertices of the polygon the points  $a_k = k$ on the real line, $k=1,\dots, N$, and choose the angles $\alpha_k = \frac2N$ for $k=1,\dots, N$. Then for $x\in \R$, 
\begin{equation}\label{eq:psidxb}
|\psi'(x+ib)| = |(x-1)+ib|^{-\frac2N} \dots |(x-N)+ib|^{-\frac2N},
\end{equation}
For $N\ge 4$ and integers $A,B$ with $1\le A<A+1<B<B+1\le N$, we will show that there is a local minimizer of $\tilde W$ in the square
$(A,A+1)\times (B,B+1)$, by showing that $\tilde W$ does not take its minimum on the boundary. 

For $x\in[1,N]$ we have, estimating $|(x-k)+ib|\le |N+ib|$ for $k=1,\dots, N$ and assuming $0<b<1$,
\[
|x-k+ib|^{-2}\ge \frac1{N^2+b^2}\ge \frac1{N^2+1}. 
\]
If $x=A\in[1,N]\cap\NN$ is an integer, one of the factors in \eqref{eq:psidxb} has the value $|b|^{-\frac2N}$, hence we find
\[
|\psi'(A+ib)|\ge b^{-\frac2N} (N^2+b^2)^{\frac1N-1}\ge b^{-\frac2N} (N^2+1)^{\frac1N-1}.
\]
At the same time $|A+\frac12+ib-k|\ge \frac12$ for every $k$ so
\[
\left|\psi'(A+\frac12+ib)\right| \le \left((\frac12)^{-\frac 2 N}\right)^N =4. 
\]
This means that the expression
\[
F(P,Q)=|P-Q|^2 |\psi'(P+ib)||\psi'(Q+ib)|
\]
satisfies 
for $(P,Q)\in \{A,A+1\} \times [B,B+1] \cup [A,A+1]\times \{ B, B+1\}=\partial ((A,A+1)\times(B,B+1))$ the bound
\[
F(P,Q) \ge  b^{-\frac2N} (N^2+1)^{\frac1N-2}
\] 
while for $P=A+\frac12$, $Q=B+\frac12$ we have 
\[
F(P,Q)\le 16 N^2 .
\]
Choosing $b$ sufficiently small such that 
\[
16N^2 < b^{-\frac2N} (N^2+1)^{\frac1N-1}
\]
 (clearly we have to choose $b=b(N)\to 0$ as $N\to\infty$) , it follows that 
\[F(A+\frac12,B+\frac12)<\min_{(P,Q)\in\partial((A,A+1)\times(B,B+1))} F(P,Q).\] 
Hence $F$ must have a local minimum in $(A,A+1)\times (B,B+1)$. 
Clearly any local minimimum of $F$ is a   local minimum of  $\tilde W$ and yields a corresponding local minimum for $W$.
From the identity theorem applied to the analytic function $F$, it follows that this minimum must be isolated.
Counting the possible choices of $A$ and $B$ that satisfy $1\le A$, $A+1<B$ and $B\le N-1$, 
we find $(N-3)+(N-4)+\dots +1 = \frac12(N-2)(N-3)$  local minimizers $(P,Q)$ of $F$ in $\R\times \R$
with $P<Q$. Each of these provides an isolated local minimizer of 
 $W_\Omega$ in 
$\partial\Omega\times\partial\Omega$. Choosing $N$ sufficiently large we then find as many 
local minimizers as desired. 
\end{proof}

\begin{remark}
The previous argument can be adapted to find smooth convex domains close to a \emph{regular} polygon such that 
the renormalized energy has many local minimizers. A Schwarz-Christoffel type formula for a 
conformal map from the unit disk to a regular $N$-gon is given by 
\[
\psi(z)=\int_0^z \frac{dw}{(1-w^N)^{2/N}},
\]
and we see that $|\psi'(z)|\to \infty$ as $z$ tends to an $N$th root of unity.
Taking $\psi(B_r(0))$ for suitably large $N$ and $r$ close to $1$, we can then construct 
smooth convex domains where the renormalized energy has many local minimizers. 
By choosing $N$ large enough, we can make this domain very close to 
a circle in the sense that its inradius is as close to its outradius as we wish. 
\end{remark}

\appendix
\section{Computing the value of $C_f$ in a special case}
For the case $f(u)=\frac1{\pi a} \sin(\pi u)$ 
 we can compute the constant $C_f$ 
explicitly, following a related computation in \cite{IK_jac}.
\begin{theorem}\label{thm:Cf}
For the nonlinearity $f(u)=f^a(u)=\frac1{\pi a} \sin(\pi u)$ (corresponding to the potential 
$G(u)=\frac{2}{a\pi^2}\cos^2(\frac{\pi u}{2})$), we have 
\begin{equation}\label{eq:CfT}
C_{f^a}=\frac2\pi \left( 1-\log a -a \log 2\right) 
\end{equation}
\end{theorem}
\begin{proof}
It is a straightforward computation to see that the increasing layer solution with $U^a(0,0)=0$ is given by 
\[U^a(x,y)=\frac2\pi\arctan{\frac{x}{y+a}}.\]
Indeed $\Delta U^a=0$ and $-U^a_y(x,0)= \frac2\pi \frac{x}{a^2+x^2}$,
while the identity $\sin \arctan t = \frac{2t}{1+t^2}$
yields
 $f^a(U^a)(x,0) =\frac1{\pi a} \frac{2ax}{a^2+x^2}$.
From Proposition~\ref{prop:i_er}, we can then compute 
\[
C_{f^a} = \lim_{R\to\infty} \left(\frac12 \int_{B_R^+} |\nabla U^a|^2 \,dxdy + \int_{-R}^R  
\frac{2}{a\pi^2}\cos^2(\frac{\pi U^a}{2})dx-\frac2\pi \log R\right).
\]
We find by direct computation
\[
|\nabla U^a(x,y)|^2 = \frac4{\pi^2} \frac1{x^2+(y+a)^2}.
\]
Shifting the integration domain, we see that 
\[
\int_{B_R^+} |\nabla U^a(x,y)|^2dxdy = \frac4{\pi^2} \int_{\{y>a, |x^2+(y-a)^2|<R^2\} } \frac1{x^2+y^2}dxdy.
\]
Now 
\[
B_R^+\cap\{ y>a\}\subset \{y>a, |x^2+(y-a)^2|<R^2\} \subset B_{R+a}^+\cap \{ y>a\}
\]
Using polar coordinates $x=s\cos \theta$, $y=s \sin \theta$,  
\[
\int_{B_R^+ \cap \{ y>a\}} \frac1{x^2+y^2} dxdy = \int_a^R  \int_{\arcsin\frac{a}s}^{\pi-\arcsin \frac{a}s} 
\frac1s d\theta ds = \int_a^R \left( \frac{\pi - 2\arcsin \frac{a}s}{s} \right)ds.
\]
Changing variables $s=\frac{a}{\sin t}$, we have $ds/s=-a\cot t dt$. It yields
\[
\int_{B_R^+ \cap \{ y>a\}} \frac1{x^2+y^2} dxdy = \pi \log\frac{R}{a} - 2 a\int_{\arcsin \frac{a}{R}}
^{\frac\pi2} t\cot t dt
\]
As $R\to\infty$, we have $\log\frac{R+1}{a}-\log \frac{R}{a} = \log(1+\frac1R)=O(\frac1R)$ 
and (since we have $x\cot x \to 1$ as $x\to 0$)
\[
\int_{\arcsin \frac{a}{R}}
^{\frac\pi2} t\cot t dt = \int_{0}
^{\frac\pi2} t\cot t dt - \int_0^{\arcsin \frac{a}{R}} t\cot t dt  = \int_{0}
^{\frac\pi2} t\cot t dt -O(\frac1R).
\]
Now integrating by parts,
\[
\int_{0}^{\frac\pi2} t\cot t dt  = -\int_0^{\frac\pi2}\log \sin t dt=\frac\pi 2 \log 2,
\]
where the final equality can be shown as follows: By symmetry,
\[
I=\int_0^{\frac\pi 2} \log \sin x dx =\int_{\frac\pi2}^{\pi} \log \sin x dx =
\int_0^{\frac\pi2} \log \cos x dx. 
\]
Hence 
\begin{align*}
2I &= \int_0^{\frac\pi2} \log(\sin x \cos x) dx = \int_0^{\frac\pi2} \log(\frac12\sin 2x)dx 
= \frac12\int_0^\pi \log(\frac12\sin y) dy \\ &= \frac\pi 2 \log\frac12+ \frac12 
\int_0^{\pi} \log\sin y dy \\ &= \frac\pi 2 \log\frac12+I.
\end{align*}

Hence
\begin{align*}
\frac12\int_{B_R^+} |\nabla U^a(x,y)|^2dxdy  & = \frac2{\pi^2} (\pi \log \frac R a  - \pi a \log 2 )+O(\frac1R)\\
&=\frac2\pi \log R -\frac2\pi \log a - \frac{2a}{\pi} \log 2 + O(\frac1R).
\end{align*}
We now compute
\[
\int_{-R}^R G(U^a)dx =\frac{2}{a\pi^2}\int_{-R}^R\cos^2(\frac{\pi U^a}{2})dx
\]
As $\cos^2\arctan t = \frac1{1+t^2}$,  
\[
\cos^2(\frac{\pi U^a}{2}) = \frac{1}{1+(\frac{x}{a})^2} = \frac{a^2}{a^2+x^2}.
\]
Hence 
\[
\int_{-R}^R G(U^a)dx  = \frac{2a}{\pi^2}\int_{-\infty}^\infty \frac{1}{a^2+x^2}dx   -O(\frac1R) 
=\frac2\pi-O(\frac1R). 
\]
Combining everything, we conclude that
\[
C_{f^a}= \frac2\pi \left( 1-\log a -a \log 2\right) = \frac2\pi \log \frac{e}{a 2^a}.
\]
\end{proof}

\bibliographystyle{siam}
\small
\bibliography{cck2}

\end{document}